\documentclass[a4paper,12pt]{article}

\usepackage{pstricks, graphicx,psfrag, overpic, amsmath, amssymb, amsthm, color,cite,url}

      \newcommand {\al}   {\alpha}          
      \newcommand {\gam } {\gamma}          \newcommand {\Gam}  {\Gamma}
                
              \newcommand {\ve}   {\varepsilon}
                 \newcommand {\vphi} {\varphi}
      \newcommand {\lam}  {\lambda}         \newcommand {\Lam}  {\Lambda}
               
      \newcommand {\mumu}   {\vartheta}       \newcommand {\nnn}   {n}
                
      \newcommand {\pl}   {\partial}

            \newcommand {\III}  {{\cal I}}
      \newcommand {\RRR}  {{\mathbb R}}

      \newcommand {\FFF}  {{\cal F}}

     \newcommand {\kap}  {{g}}                           
       \newcommand {\DDD}  {\big[D\big]}      \newcommand {\nvis}  {\mathcal{F}}
     \newcommand {\beq}  {\begin{equation}}
      \newcommand {\eeq}  {\end{equation}}

    \newcommand {\PP}{\Phi}
      \newtheorem{theorem}{Theorem}
      \newtheorem{lemma}{Lemma}
      \newtheorem{utv}{Proposition}
      \newtheorem{zam}{Remark}

\title{The problem of optimal camouflaging}

\author{Alexander Plakhov\thanks{Center for R{\&}D in Mathematics and Applications, Department of Mathematics, University of Aveiro, Portugal,
plakhov@ua.pt} \and Vera Roshchina\thanks{School of Mathematics and Statistics, UNSW Sydney, Australia, v.roshchina@unsw.edu.au}}

\begin{document}

\maketitle

\begin{abstract}
We consider the problem of camouflaging for bodies with specular surface in the framework of geometric optics. The {\it index of visibility} introduced in \cite{camouflage} measures the mean deviation of light rays incident on the body's surface. We study the problem of minimal visibility index for bodies with fixed volume contained in the unit sphere. This problem is reduced to a special vector-valued problem of optimal mass transfer, which is solved partly analytically and partly numerically. This paper is a continuation of the study started in \cite{UMN09}, \cite{camouflage}, and \cite{invisibility}.
\end{abstract}

\begin{quote}
\end{quote}

\begin{quote}
\end{quote}

\section{Introduction}

\subsection{Preliminary notes}
To hide from human eyes, to become invisible or almost invisible --- this idea goes back to the folk epics, like Perseus' magic helmet of invisibility, it occupies the mind of the scientist from H. G. Wells' novel {\it The Invisible Man}, the developers of military equipment (stealth technology), and enthralls the masses via popular books and Hollywood movies, from Predator to Harry Potter. The ideas of invisibility and camouflage have many faces and allow for many different formulations in physics, technology and mathematics (bending light around an object using metamaterial cloaking).

In this article, we consider the idea of invisibility and camouflaging for objects with a mirror reflective surface within the framework of geometric optics, which in addition to the undeniable entertainment value has deep connections with fundamental phenomena and unsolved problems in dynamical systems. From the mathematical point of view, we study billiards in the exterior of bounded sets, and so `invisibility' refers to the construction of bodies that have reflective properties that either mimic invisibility or minimise resistance in terms of a certain natural measure. In \cite{invisibility} it was proved that there are no perfectly invisible sets in such model: whatever the shape of such an object, a careful observer will find that the background looks different with the presence of the object than without it. However the question remains on how close can we get to invisibility.


It is natural to introduce a value measuring visibility of the body (index of visibility), and to find the body, from a certain class of admissible bodies, for which this value is minimal. One approach is to use the index of visibility that was introduced in \cite{camouflage}. The problem of minimization for this index, however, seems to be quite difficult in general, so we focus on a more accessible task of finding a reasonable lower estimate for the visibility index. Namely, we consider a class of sets with fixed volume in $\RRR^d,\, d \ge 2$ contained in a ball of fixed radius, and show that the visibility index in this class is always greater than a certain value.

An estimate of this kind has already been obtained in \cite{camouflage}. In addition, in \cite{UMN09} an estimate was found in the case when the volume of the set is close to the volume of the ambient ball. In the present paper we obtain improved estimates, which are significantly sharper than the one in \cite{camouflage} and coincide with the estimate in \cite{UMN09} in the limit when the volume of the set tends to the volume of the ambient ball (see Theorems~\ref{tt1}, \ref{tt2}, and \ref{t5}). Our method is based on solving a particular problem of optimal mass transfer with a constraint. In addition to the theoretical one, a numerical estimate is also obtained. We do not yet know if the lower bound obtained here is sharp.

Our paper is organized as follows. In Sections~\ref{sec:maths} and~\ref{sec:results} of this Introduction, we introduce the mathematical context of our problem and give an overview of our results. The rest of the paper is dedicated to technical details and proofs: in Sections~\ref{sec:proofThm1} and~\ref{sec:proofThm2} we prove our main theoretical results (Theorems~\ref{tt1} and~\ref{tt2} respectively), and Section~\ref{sec:numeric} is dedicated to Theorem \ref{t5} and the description and outcomes of our numerical experiments. There is also a brief Appendix that contains technical calculations that are required for mathematical rigour, but are less helpful for the intuition and for the big picture understanding of our findings.


\subsection{Mathematical formulation}\label{sec:maths}
Consider a compact finitely connected set $D$ with piecewise smooth boundary in the Euclidean space $\RRR^d,\, d \ge 2$ and a compact convex set $C$ 
containing $D$. Thus, we have $D \subset 
C \subset \RRR^d$. Denote by $n_\xi$ the unit outward normal to $C$ at a regular point\footnote{A point on the boundary of a convex set $C \subset \RRR^d$ is regular, if there is only one hyperplane of support (line if $d=2$ and plane if $d=3$) through this point.} $\xi \in \pl C$. By $\langle \cdot \,, \cdot \rangle$ we designate the scalar product.

Denote by $\mumu = \mumu_C$ the measure on $S^{d-1} \times \pl C$ such that $d\mumu(v,\xi) = \langle v, n_\xi \rangle_+ dv\, d\xi$. Here and in what follows, $z_+$ and $z_-$ mean the positive and negative parts of $z$, that is, $z_\pm = \max\{ 0, \pm z \}$. It follows that the $\mumu$-measure of the set $(S^{d-1} \times \pl C) \cap \{ \langle v, n_\xi \rangle < 0 \}$ is zero.

Consider the billiard in $\RRR^d \setminus D$ and define the map $(v,\xi) \mapsto (v^+, \xi^+)$ 
from a subset of $S^{d-1} \times \pl C$ to $S^{d-1} \times \pl C$ as follows. For each $(v,\xi) \in S^{d-1} \times \pl C$ consider the billiard particle that starts moving at the time instant $t = 0$ at the point $\xi$ with the velocity $-v$, and fix the instant $t \ge 0$ when it leaves $C$. Denote by $\xi^+ = \xi^+_{D,C}(v,\xi)$ and by $v^+ = v^+_{D,C}(v,\xi)$ the position and the velocity of the particle at that instant; see Fig.~\ref{fig:CD}.\footnote{If the particle at some instant gets into a singular point of $\pl D$, or makes infinitely many reflections on a finite time interval, the values $\xi^+_{D,C}(v,\xi)$ and $v^+_{D,C}(v,\xi)$ are not defined. The set of corresponding points $(v,\xi) \in S^{d-1} \times \pl C$ has measure zero; for the proof see, e.g., Section 1.7 in the book \cite{T}.}
\begin{figure}
\begin{picture}(0,110)
\rput(8,2){
\scalebox{1}{
\pscircle(0,0){2}
\psecurve[fillstyle=solid,fillcolor=lightgray](0.5,-0.63)(-0.2,-1.25)(1,-1)(0.75,0.2)(0.5,-0.63)(-0.2,-1.25)(1,-1)
\psecurve[fillstyle=solid,fillcolor=lightgray](-0.1,-0.4)(0,0.45)(0.67,1.1)(-0.5,1)(-0.4,-0.4)(-0.1,-0.4)(0,0.45)(0.67,1.1)
\psline[linecolor=red,arrows=->,arrowscale=2,linewidth=0.8pt](-1.9,-0.63)(-0.9,-0.63)
\psline[linecolor=red,arrows=->,arrowscale=2,linewidth=0.8pt](-1.9,-0.63)(-2.75,-0.63)
\psline[linecolor=red,arrows=->,arrowscale=2,linewidth=0.8pt](-0.9,-0.63)(0.5,-0.63)(0,0.45)(1.8,0.87)
\psline[linecolor=red,arrows=->,arrowscale=2,linewidth=0.8pt](1.8,0.87)(2.7,1.08)
\rput(2.6,1.4){\scalebox{1}{$v^+$}}
\rput(-2.15,-0.3){\scalebox{1}{$\xi$}}
\rput(2.23,0.7){\scalebox{1}{$\xi^+$}}
\rput(-2.75,-0.88){\scalebox{1}{$v$}}
\psdots(-1.9,-0.63)(1.8,0.87)
\rput(2.25,-0.6){\scalebox{1.2}{$C$}}
}
}
\end{picture}
\caption{The domain $D$ (with 2 connected components) is shown shaded, and the ambient body $C$ is the disk.}
\label{fig:CD}
\end{figure}
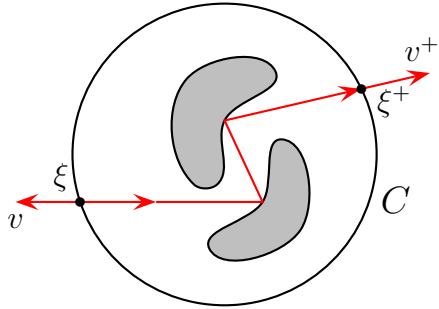
The map
$$
T = T_{D,C} : \ (v, \xi) \mapsto (v^+_{D,C}(v,\xi),\ \xi^+_{D,C}(v,\xi))
$$
is a one-to-one measure preserving map from a full-measure subset of the measure space $(S^{d-1} \times \pl C,\, \mumu)$ onto itself.

The value
\beq\label{RD}
\mathcal{R}_C(D) = \int_{S^{d-1} \times \pl C} \frac 12\, |v + v^+_{D,C}(v,\xi)|^2\, d\mumu(v,\xi)
\eeq
is called the {\it index of visibility} of $D$. It is proved in the book \cite{ebook}, Chapter 2, that this value does not depend on the convex body $C$ containing $D$, and therefore, one can write $\mathcal{R}(D)$ in place of $\mathcal{R}_C(D)$.

The interpretation of this value in terms of geometric optics is the following. The integrand $\frac 12\, |v + v^+_{D,C}(v,\xi)|^2 = \langle v + v^+_{D,C}(v,\xi), \ v \rangle$ measures the deviation of the direction of a light ray incident on $D$. Correspondingly, $\mathcal{R}(D)$ is the sum of deviations of the direction of all incident light rays, and thus, can serve as a measure of visibility of $D$.

This value can also be interpreted in terms of classical mechanics. Consider a solid body $D$ moving with unit velocity $v$ in a rarefied medium composed of point particles at rest. The particles do not mutually interact and are reflected from the body in the perfectly elastic way. This is the mechanical model considered by Newton in his {\it Principia} and studied by many authors afterwards (see, e.g., \cite{BK,BFK,BrFK,CL1,CL2,LO,LP1,UMN09,canadian,ebook,SIREV,Nonl2016}).

The {\it resistance} of $D$ in the direction $v$ is the modulus of the orthogonal projection of the drag force on the direction of motion $v$. Assuming that the medium density is $1$, the resistance in the direction $v$ equals
$$
R(D,v) = \int_{\{ \xi \in \pl C : \langle v, n_\xi \rangle \le 0\} } \langle v + v^+_{D,C}(v,\xi),\, v \rangle\, |\langle v, n_\xi \rangle|\, d\xi.
$$
The function $v \mapsto R(D,v)$ exists for almost all $v \in S^{d-1}$, is measurable and bounded, and therefore, integrable. We have
$$
\mathcal{R}(D) = \int_{S^{d-1}} R(D,v)\, dv.
$$

Thus, the value $\mathcal{R}(D)$ is the average of resistance in a direction over all directions. Therefore it can also be called the {\it mean resistance} of $D$.

Let $B_r^d$ designate the ball with radius $r$ centered at the origin. Denote by $s_{d-1} = \frac{2\pi^{d/2}}{\Gam(d/2)}$ the area of the $(d-1)$-dimensional unit sphere $S^{d-1}$, and by $b_{d} = \frac{2\pi^{d/2}}{d\Gam(d/2)} = \frac{1}{d} s_{d-1}$ the volume of the $d$-dimensional unit ball $B_1^d$. One has, in particular, $s_0 = 2,\; s_1 = 2\pi,\; s_2 = 4\pi,\; b_1 = 2,\; b_2 = \pi,\; b_3 = 4\pi/3$.

Slightly abusing the notation, we designate by $|D|$ the $d$-dimensional volume of the set $D \subset \RRR^d$, and by $|\pl C|$, the $(d-1)$-dimensional volume of the boundary of the convex set   
$C \subset \RRR^d$.

Note that the index of visibility of a {\it convex} set $C \subset \RRR^d$ can easily be determined (see sections 6.1.1 and 6.2 of \cite{ebook}). For the reader's convenience we provide the derivation of $\mathcal{R}(C)$ here. One has $\xi_{C,C}^+(v,\xi) = \xi$,\, $v_{C,C}^+(v,\xi) = -v + 2\langle v,\, n_\xi \rangle n_\xi$, hence
$$
\mathcal{R}(C) = \int_{S^{d-1} \times \pl C} 2\, \langle v,\, n_\xi \rangle^2\, d\mumu(v,\xi) = |\pl C| \int_{S^{d-1}} 2\, \langle v,\, n_{\xi_0} \rangle_+^3\, dv,
$$
where $\xi_0$ is an (arbitrarily chosen) regular point on $\pl C$.
Take the map $v \mapsto v - \langle v,\, n_\xi \rangle n_\xi = w = r\gam$ (with $r=|w|$ and $\gam \in S^{d-2}$), which is the orthogonal projection of $S^{d-1}$ onto the unit $(d-1)$-dimensional ball in the plane $\{ n_\xi \}^\perp$. We have $\langle v,\, n_\xi \rangle_+\, dv = dw = d\gam\, dr$, and
$$
\int_{S^{d-1}} \langle v,\, n_\xi \rangle_+^3\, dv = \int_{B_1^{d-1}} (1 - r^2)\, d\gam\, dr
= s_{d-2} \int_0^1 (1 - r^2)\, r^{d-2}\, dr = \frac{2}{d^2 - 1}\, s_{d-2} = \frac{2}{d + 1}\, b_{d-1},
$$
hence $\mathcal{R}(C) = \frac{4}{d+1}\, b_{d-1} |\pl C|$.
In the 2D and 3D cases one has, respectively, $\mathcal{R}(C) = \frac{8}{3} |\pl C|$ and $\mathcal{R}(C) = \pi |\pl C|$.  Thus, the  index of visibility of the unit ball $B_1^d$ equals
$$\mathcal{R}(B_1^d) = \frac{4b_{d-1}s_{d-1}}{d+1},$$
and in particular, in the 2D and 3D cases
$$\mathcal{R}(B_1^2) = \frac{16}{3}\, \pi \quad \text{and} \quad \mathcal{R}(B_1^3) = 4\pi^2.$$

It is convenient to introduce the normalized volume $\DDD$ and the normalized visibility $\nvis(D)$  of a body $D$ as
\beq\label{norm}
\DDD := \frac{1}{b_d}\, |D|, \qquad \nvis(D) := \frac{d+1}{4b_{d-1}s_{d-1}}\, \mathcal{R}(D),
\eeq
so that the normalized volume and normalized visibility of the unit ball are equal to 1.

\subsection{Results}\label{sec:results}

In this paper we are concerned with the problem of minimal index of visibility for bodies with fixed volume contained in a unit sphere.
In terms of normalized volume and normalized visibility the problem takes the form
\beq\label{infi}
\text{Find} \quad \inf \{ \FFF(D) : \ \DDD = a \text{ and } D \subset B_1^d \}, \quad \text{where } 0 < a < 1.
\eeq

We cannot yet solve it completely. The aim of this paper is to derive several lower estimates for the infimum.

\begin{zam}
Note that the index of visibility $\FFF(D)$ is invariant under translations and possesses the homogeneity property: $\FFF(rD) = r^{d-1} \FFF(D)$, where $rD := \{ rx : x \in D \}$. This means that having found a lower bound for the infimum in \eqref{infi}, we will be able to provide a similar lower bound for the class of sets with fixed volume contained in any fixed ball.
\end{zam}

Here we continue the work initiated in the papers \cite{UMN09} (see also Chapters 5 and 6 of \cite{ebook}) and \cite{camouflage}.

Let $\varsigma = \varsigma_d$ be the probability measure on $[0,\, \pi/2]$ with the generating function $\sin^{d-1}\vphi$, that is, 
\begin{equation}\label{varsigma}
d\varsigma(\vphi) = (d-1) \sin^{d-2}\vphi \cos\vphi\, d\vphi.
\end{equation}
Denote by $\Gam_{\varsigma,\varsigma}$ the set of measures $\chi$ defined on the square $[0,\, \pi/2]^2$ such that both marginals of $\chi$ coincide with $\varsigma$. In other words, $\Gam_{\varsigma,\varsigma}$ is the set of measures $\chi$ on $[0,\, \pi/2]^2$ satisfying the following condition: for any Borel set $A \subset [0,\, \pi/2]$,
$$
\chi(A \times [0,\, \pi/2]) = \chi([0,\, \pi/2] \times A) = \varsigma(A).
$$

In Section 4 of \cite{UMN09} the following problem was studied: for a convex body $C \subset \RRR^d$ determine the limit
$$
m_d(C) := \lim_{\ve\to 0^+}  \frac{\inf\{ \FFF(D) : \, D \subset C,\ |C \setminus D| < \ve \}}{\FFF(C)}.
$$
It was proved that these values do not depend on $C$, and each value $m_d(C) = m_d$ is the minimum in the problem
\beq\label{md}
m_d =  \dfrac{d+1}{4} \inf_{\chi\in\Gam_{\varsigma,\varsigma}}  \int_{[0,\, \pi/2]^2} (1 + \cos(\vphi+\psi))\, d\chi(\vphi,\psi)
\eeq
(see formula (54) in section 4.2 of \cite{UMN09}). This minimum was proved to be the minimum of a certain well-defined function of one variable, and then was numerically determined. It was found, in particular, that $m_2 = 0.987820...$ and $m_3 = 0.969445...$.

This result can be interpreted as follows. A sequence of sets $D_n \subset C$, such that $|C \setminus D_n| \to 0$ as $n \to \infty$ and there exists $\lim_{n\to\infty} \FFF(D_n)$, is called a {\it roughening} of $C$. The aforementioned limit $\lim_{n\to\infty} \FFF(D_n)$ is called the resistance of the roughening. It is proved that the minimal resistance of a roughening of $C$ equals $m_d\, \FFF(C)$, independently of the choice of $C$.

In particular, taking $C$ to be the unit ball $B_1^d$, one can find a sequence of sets $D_n$ contained in the ball with their volumes approaching the ball's volume, such that
$$
\lim_{n\to\infty} \FFF(D_n) = m_d,
$$
and the limit cannot be made smaller. In other words,
$$
\lim_{\ve\to0^+} \inf_{\substack{D\subset B_1^d\\ [D] = 1-\ve}} \FFF(D) = m_d.
$$

In \cite{camouflage} the following Theorems \ref{t1} and \ref{t2} are obtained. They are stated here in a slightly modified form taking into account the different (normalized) notation adopted here.

The positive values $c_d$ in Theorem \ref{t1} are specified in \cite{camouflage} (in particular, $c_2 = \pi^3/288 \approx 0.11$ and $c_3 = 16/729 \approx 0.022$).

\begin{theorem}\label{t1} {\rm\bf (Theorem 2.1 in \cite{camouflage})}

Let $D \subset \RRR^d$ be contained in a ball of radius $1$. Then
$$
\FFF(D) \ge h_d(\DDD),
$$
where $h_d$ is a function of positive variable satisfying
$$
h_d(x) = c_d x^3 (1 + o(1)) \quad \text{as} \ \, x \to 0^+.
$$
\end{theorem}

In the 2D and 3D cases the following stronger theorem is valid.

\begin{theorem}\label{t2} {\rm\bf (Theorem 2.3 in \cite{camouflage})}

(a) Let $D \subset \RRR^2$ be contained in a circle of radius $1$. Then
$$
\FFF(D) \ge \frac{\pi^3}{288} \DDD^3  \approx 0.11\, \DDD^3.
$$

(b) Let $D \subset \RRR^3$ be contained in a ball of radius $1$. Then
$$
\FFF(D) \ge \frac{16}{729} \DDD^3  \approx 0.022\, \DDD^3.
$$
\end{theorem}

These estimates are disappointingly small, especially when the values of $\DDD$ are close to 1.

The main results of the present paper are Theorems \ref{tt1}, \ref{tt2}, and \ref{t5}, which provide a significant strengthening of Theorems \ref{t1} and \ref{t2}. The proofs of Theorems \ref{tt1} and \ref{tt2} are given in Sections~\ref{sec:proofThm1} and \ref{sec:proofThm2} respectively. Theorem \ref{t5} is stated and proved in Section \ref{sec:numeric}.

\begin{theorem}\label{tt1}
Let a set $D \subset \RRR^d$ be contained in a unit ball. Then
\beq\label{est}
\nvis(D) \ge m_d - \frac{d+1}{4}\, \frac{b_d}{b_{d-1}}\, (1 - \DDD).
\eeq
In particular,
\begin{equation*}
\begin{split}
& \nvis(D) \ge m_2 - \frac{3\pi}{8} \big(1 - \DDD\big) \ \ \, \text{for $d=2$} \\
\text{and} \quad & \nvis(D) \ge m_3 - \frac{4}{3} \big(1 - \DDD\big) \ \ \, \text{for $d=3$}.
\end{split} \end{equation*}
\end{theorem}

\begin{theorem}\label{tt2}
Let a set $D \subset \RRR^d$ be contained in a unit ball. Then
\begin{equation*}
\begin{split}
\text{\rm (a)} \ \  \nvis(D) &\ge \dfrac{1}{2c}\, \DDD^2, \quad \text{where}  \\
c &= \dfrac{8}{d+1}  \Big( \dfrac{b_{d-1}}{b_{d}} \Big)^2\Big[ \dfrac{b_d}{b_{d-1}} \Big( 1 - \dfrac{1}{\sqrt{2}} \Big) + \sqrt{2}\,\dfrac{d-1}{d} + \dfrac{\pi}{2} \Big( \dfrac{\pi}{4} - \dfrac{1}{\sqrt{2}} \Big) \Big]; \\ 
\text{\rm (b)} \ \  \nvis(D) &\ge \dfrac{d(d+1)}{16(d - 1)}\, \Big( \dfrac{b_{d}}{b_{d-1}} \Big)^2 \DDD^2 (1 + o(1)) \quad \text{as} \ \ \DDD \to 0.
\end{split}
\end{equation*}
\end{theorem}

In particular
$$
\nvis(D) \ge 0.358\, \DDD^2 \ \ \, \text{for $d=2$ \ \ and} \ \ \ \nvis(D) \ge 0.367\, \DDD^2 \ \ \, \text{for $d=3$},
$$ 
and additionally, when $\DDD \to 0$,
$$
\nvis(D) \ge \dfrac{3\pi^2}{32}\, \DDD^2 (1 + o(1)) \ \ \, \text{for $d=2$ \ \ and} \ \ \ \nvis(D) \ge \dfrac{2}{3}\, \DDD^2 (1 + o(1)) \ \ \, \text{for $d=3$}.
$$

\begin{figure}[ht]
\begin{overpic}[
width=0.45\textwidth]{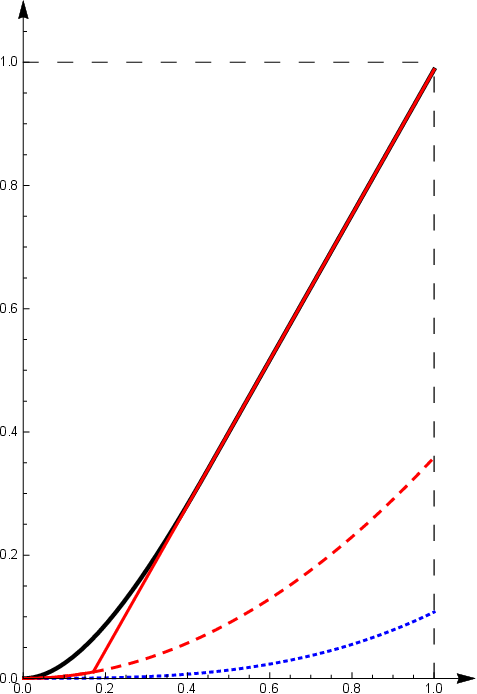}
\put(7,95){$\inf\nvis$}
\put(66,6){$[D]$}
\put(35,70){$d=2$}
\end{overpic}
\qquad
\begin{overpic}[
width=0.45\textwidth]{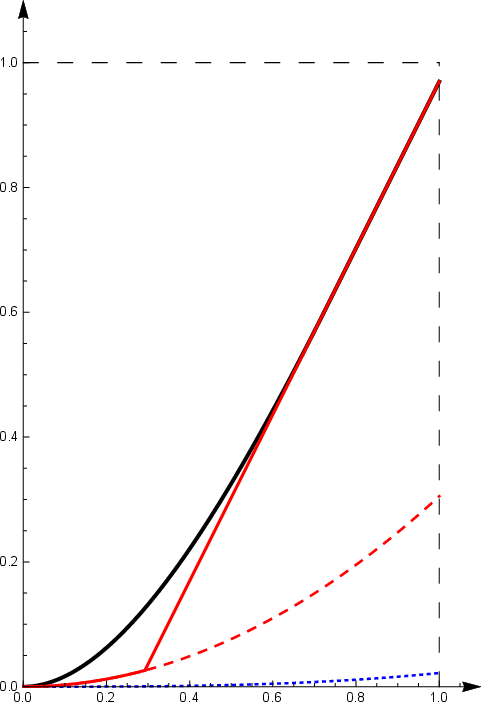}
\put(7,95){$\inf\nvis$}
\put(66,6){$[D]$}
\put(35,70){$d=3$}
\end{overpic}

\caption{The lower estimate for $\inf\nvis$ obtained in \cite{camouflage} (Theorem \ref{t2}) is the dotted (blue) line, while the lower estimate given by Theorems \ref{tt1} and \ref{tt2} combined is the thin solid (red) line. The lower estimate from Theorem~\ref{tt2} is the dashed (red) line. The bound obtained from the numerical approximation based on Theorem \ref{t5} is shown in thick solid black.
}
\label{figGraphs}
\end{figure}
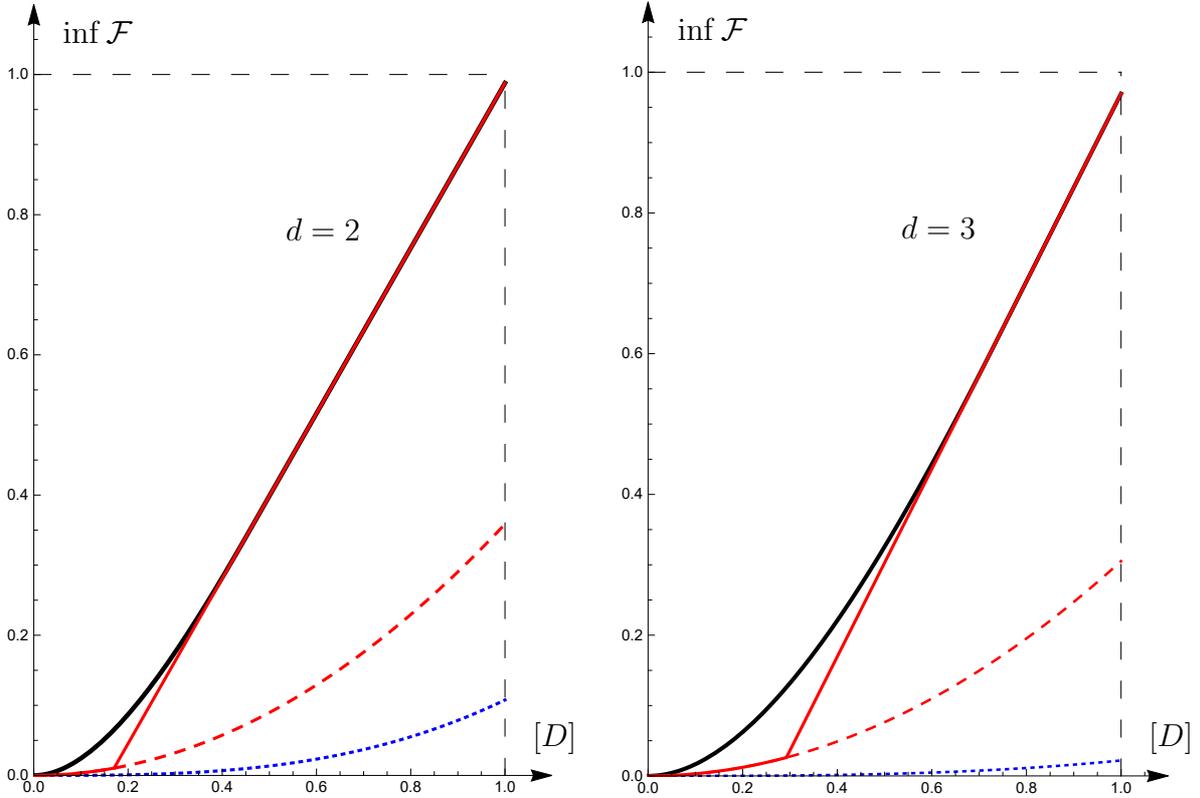

\begin{zam}
Theorems \ref{tt1} and \ref{tt2} together provide much better estimates of the index of visibility that Theorems \ref{t1} and \ref{t2}. For $\DDD$ small the estimates are proportional to $\DDD^2$, rather than $\DDD^3$, and in the limit $\DDD \to 1$ the estimates coincide with those found in \cite{UMN09} and are approximately 9 times (for $d=2$) and 44 times (for $d=3$) greater than those obtained in Theorems \ref{t1} and \ref{t2}.
\end{zam}

\begin{zam}\label{zamopen}
There is an interesting open question:
\begin{quote}
{\rm Is it true that} \quad $\inf_{[D]=1} \FFF(D) = 0?$
\end{quote}
Note that we do not impose here the condition that the admissible bodies $D$  are contained in a fixed ball. (Recall, however, that they are bounded.) In other words, can one achieve arbitrarily small visibility of a certain mass by re-distributing it in a sufficiently large space?

Positive answer to this question would mean that there exists a family of bodies $D_r \subset B_r^d$,\, $r > 0$ with $\big[ D_r \big] = 1$ such that $\FFF(D_r) \to 0$ as $r \to \infty$. Each body $D_r/r$ is then contained in the unit ball $B_1^d$ and additionally,
$$
\Big[ \frac{D_r}{r} \Big] = \frac{1}{r^d} \qquad \text{and} \qquad \FFF\Big( \frac{D_r}{r} \Big) = \frac{\FFF(D_r)}{r^{d-1}},
$$
and therefore,
$$
\inf_{[D]=a,\, D\subset B_1^d} \FFF(D) = o(a^{\frac{d-1}{d}}) \qquad \text{as} \ \ a \to 0.
$$
However, the only upper bound for the infimum we are able to provide at the moment, is the trivial one:
$$\inf_{[D]=a,\, D\subset B_1^d} \FFF(D) \le m_d \, a^{\frac{d-1}{d}}.
$$
\end{zam}

In addition to the theoretical estimates in Theorems~\ref{tt1} and~\ref{tt2}, we also performed basic numerical experiments, in which a discretised version of the problem \eqref{infi} is solved using a linear programming formulation. The numerical solutions that we obtained are shown in Fig.~\ref{figGraphs}. We explain our approach in Section~\ref{sec:numeric}, where we also discuss the solutions to the linear programming problem that approximate the optimal measure $\chi$.

\section{Proof of Theorem \ref{tt1}}\label{sec:proofThm1}

\subsection{Auxiliary problem of optimal mass transfer}

The phase volume (with respect to the standard Liouville measure $dv dx$) of the billiard in $C \setminus D$ equals
$$
V = |S^{d-1} \times (C \setminus D)| = s_{d-1} (|C| - |D|). 
$$
Let $l(v,\xi) = l_{D,C}(v,\xi)$ be the length of the piece of the billiard trajectory with the initial velocity and initial position $(-v,\xi) \in S^{d-1} \times \pl C$ and with the final point being the point of the next intersection with $\pl C$. 
Later on we will use the so-called {\it Santal\'o-Stoyanov formula}
\beq\label{S-S}
V \ge \int_{S^{d-1}\times\pl C} l(v,\xi)\, d\mumu(v,\xi).
\eeq
This formula states that the phase volume is greater than or equal to the integral of the lengths of billiard trajectories over the initial data.
(The sign $">"$ in the formula is realized when a part of the phase space with positive measure is not accessible for particles starting from $\pl C$). 

In the particular case when $D$ is absent, $D = \emptyset$, the inequality in \eqref{S-S} becomes equality, and is called {\it Santal\'o's formula}. Its derivation and corresponding discussion and references can be found, e.g., in Section 2 of the paper \cite{Chernov} (formula (2.4) in that paper). Santal\'o's formula was generalized by Stoyanov in 
\cite{stoyanov} to the case of nonempty $D$; see Corollary 1.3 in \cite{stoyanov}.

Take $C = B_1^{d}$ and denote
\beq\label{mu}
\mu = \mu_d := \frac{1}{s_{d-1} b_{d-1}}\mumu_{S^{d-1}}.
 \eeq
One easily checks that $\mu$ is a probability measure on $(S^{d-1})^2$. By $\Gam_{\mu,\mu}$ we denote the set of measures $\nu$ on $(S^{d-1})^4 = (S^{d-1})^2 \times (S^{d-1})^2$ such that both marginals of $\nu$ coincide with $\mu$. In other words, $\Gam_{\mu,\mu}$ is the set of measures $\nu$ on $(S^{d-1})^4$ satisfying the following condition: for any Borel set $A \subset (S^{d-1})^2$,
$$
\nu(A \times (S^{d-1})^2) = \nu((S^{d-1})^2 \times A) = \mu(A).
$$
A measure $\nu \in \Gam_{\mu,\mu}$ is called a {\it transport plan} from the measure space $((S^{d-1})^2,\, \mu)$ onto itself.

Since for $\pl C = S^{d-1}$, each point $\xi \in \pl C$ coincides with $n_\xi$, from now on we write $n$ in place of $\xi$.

For any $D \subset S^{d-1}$ the map $T = T_{D,B_1^{d}}$ induces the transport plan $\nu_D \in \Gam_{\mu,\mu}$ defined as follows: for any Borel set $A \subset (S^{d-1})^4$,\ $\nu_D(A) := \mu\big( \{ (v,n) : ((v,n), T(v,n)) \in A \} \big)$.

Using \eqref{RD}, \eqref{norm}, and \eqref{mu}, one gets
\beq\label{eqF} \begin{split}
\nvis(D) &= \frac{d+1}{4} \int_{(S^{d-1})^2} \frac 12\, |v + v^+_{D,B_1^d}(v,n)|^2\, d\mu(v,n)
\\
&= \frac{d+1}{4} \int_{(S^{d-1})^4} \frac{|v + v^+|^2}{2}\, d\nu_D(v,n,v^+,n^+).
\end{split}  \eeq

Using the obvious inequality $l(v,\nnn) \ge |\nnn - \nnn_{D}^+(v,\nnn)|$, by formula \eqref{S-S} we obtain
$$
V \ge \int_{(S^{d-1})^2} |\nnn - \nnn_{D}^+(v,\nnn)|\, d\mumu(v,n) = s_{d-1} b_{d-1} \int_{(S^{d-1})^4} |\nnn - \nnn^+|\, d\nu_D(v,n,v^+,n^+).
$$
Note that the phase volume of the billiard in $B_1^d \setminus D$ equals
$$
V = s_{d-1} (b_d - |D|) = s_{d-1} b_d \big(1 - \DDD\big),
$$
hence
\beq\label{eqD}
\DDD \le 1 - \frac{b_{d-1}}{b_d} \int_{(S^{d-1})^4} |n - n^+|\, d\nu_D(v,n,v^+,n^+).
\eeq

Now consider the following auxiliary problem of optimal mass transport, which is closely connected with our problem of minimal visibility:
 \begin{quote}
\text{ for all $\lam > 0$ find} \ $\inf \{ F_2(\nu) - \lam F_1(\nu) : \ \nu \in \Gam_{\mu,\mu} \}$,
 \end{quote}
 where
\beq\label{F1}
F_1(\nu) = 1 - \frac{b_{d-1}}{b_d} \int_{(S^{d-1})^4} |n_1 - n_2|\, d\nu(v_1,n_1,v_2,n_2)
\eeq
and
\beq\label{F2}
F_2(\nu) = \frac{d+1}{4} \int_{(S^{d-1})^4} \frac{|v_1 + v_2|^2}{2}\, d\nu(v_1,n_1,v_2,n_2).
\eeq

Note in passing that for $\lam = 0$ the infimum equals zero and is attained at each admissible measure $\nu$ supported on the subspace $v_1 + v_2 = 0$.

The above problem can be viewed as a part of the following vector-valued optimal mass transport problem. Denote $F(\nu) = (F_1(\nu), F_2(\nu))$. It is required to find the convex planar set $F(\Gam_{\mu,\mu}) = \{ F(\nu) : \, \nu \in \Gam_{\mu,\mu} \} \subset \RRR^2$. This problem, however, will not be considered here in full.

By \eqref{eqF} and \eqref{eqD} one has $\DDD \le F_1(\nu_D)$ and $\nvis(D) = F_2(\nu_D)$, and therefore,
\beq\label{ineqF}
\begin{split}
\inf_{D \subset S^{d-1}} \big( \nvis(D) - \lam\DDD \big) \ge \inf_{\nu\in\Gam_{\mu,\mu}} \big( F_2(\nu) - \lam F_1(\nu) \big)\\
=  \inf_{\nu\in\Gam_{\mu,\mu}} \int_{(S^{d-1})^4} \left[ \frac{d+1}{4}\, \frac{|v_1 + v_2|^2}{2} + \lam \Big( \frac{b_{d-1}}{b_d}\, |n_1 - n_2| - 1 \Big) \right] d\nu(v_1,n_1,v_2,n_2).
\end{split}
\eeq

Now for each quadruple of vectors $v_1,\, n_1,\, v_2,\, n_2$ denote by $2\eta$ the angle between $n_1$ and $n_2$, by $\vphi$ the angle between $n_1$ and $v_1$, and by $\psi$ the angle between $n_2$ and $v_2$; see Fig.~\ref{fig:angles}.
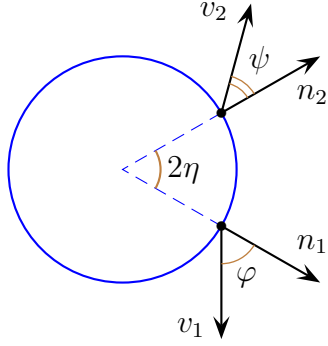
\begin{figure}
\begin{picture}(0,110)
\rput(7.5,2){
\scalebox{1}{
\pscircle[linecolor=blue](0,0){1.5}


\rput(0.9,-2.1){\scalebox{1}{$v_1$}}
\rput(1.2,2.1){\scalebox{1}{$v_2$}}
\rput(2.5,0.99){\scalebox{1}{$n_2$}}
\rput(2.5,-0.99){\scalebox{1}{$n_1$}}
\psline[linecolor=blue,linewidth=0.4pt,linestyle=dashed](1.3,0.75)(0,0)(1.3,-0.75)
\psline[arrows=->,arrowscale=2,linewidth=0.8pt](1.3,0.75)(2.6,1.5)
\psline[arrows=->,arrowscale=2,linewidth=0.8pt](1.3,0.75)(1.69,2.2)
\psline[arrows=->,arrowscale=2,linewidth=0.8pt](1.3,-0.75)(2.6,-1.5)
\psline[arrows=->,arrowscale=2,linewidth=0.8pt](1.3,-0.75)(1.3,-2.25)
\psdots(1.3,0.75)(1.3,-0.75)
\psarc[linewidth=0.5pt,linecolor=brown](1.3,0.75){0.4}{30}{75}
\psarc[linewidth=0.5pt,linecolor=brown](1.3,0.75){0.5}{30}{75}
\psarc[linewidth=0.5pt,linecolor=brown](1.3,-0.75){0.5}{-90}{-30}
\psarc[linewidth=0.8pt,linecolor=brown](0,0){0.5}{-30}{30}
\rput(0.8,0){\scalebox{1}{$2\eta$}}
\rput(1.63,-1.43){\scalebox{1}{$\vphi$}}
\rput(1.8,1.45){\scalebox{1}{$\psi$}}

}
}
\end{picture}
\caption{The vectors $n_1$,\, $n_2$,\, $v_1$,\, $v_2$.}
\label{fig:angles}
\end{figure}
That is, $\eta = \frac 12 \arccos \langle n_1, n_2 \rangle$,\, $\vphi = \arccos \langle n_1, v_1 \rangle$, and $\psi = \arccos \langle n_2, v_2 \rangle$. These angles lie in  $[0,\, \pi/2]$.
Denote
\beq\label{minK}
K_\lam(\vphi, \psi) = K_\lam^d(\vphi, \psi) := \!\!\!\! \inf_{\substack{\langle v_1, n_1\rangle = \cos\vphi\\ \langle v_2, n_2\rangle = \cos\psi\\ |v_1|=|n_1|=|v_2|=|n_2|=1}}\!\!\!\!\!
\left[ \frac{d+1}{4}\, \frac{|v_1 + v_2|^2}{2} + \lam \Big( \frac{b_{d-1}}{b_d}\, |n_1 - n_2| - 1 \Big) \right].
\eeq
An evaluation of $K_\lam$ will be given later on in formulae \eqref{Klamb}, \eqref{VkLamTheta}, and \eqref{etaLam}.

Recall that the measure $\varsigma = \varsigma_d$ is defined by formula \eqref{varsigma}.
Define the map $\Upsilon$ from $(S^{d-1})^4$ to $[0,\, \pi/2]^2$ by
$$
\Upsilon(v_1, n_1, v_2, n_2) = (\arccos \langle n_1, v_1 \rangle, \ \arccos \langle n_2, v_2 \rangle).
$$
The corresponding push-forward map takes each measure $\nu \in \Gam_{\mu,\mu}$ to the measure $\chi = \Upsilon\#\nu \in \Gam_{\varsigma,\varsigma}$, that is, $\chi$ is defined on the square $[0,\, \pi/2]^2$, and both projections of $\chi$ on the sides of the square coincide with $\varsigma$. It follows that
\beq\label{ineqVK}
\inf_{\nu\in\Gam_{\mu,\mu}} \big( F_2(\nu) - \lam F_1(\nu) \big) \ge
\inf_{\chi\in\Gam_{\varsigma,\varsigma}}  \int_{[0,\, \pi/2]^2} K_{\lam}(\vphi,\psi)\, d\chi(\vphi,\psi).
\eeq
The inequality in \eqref{ineqF} together with \eqref{minK} and \eqref{ineqVK} provides a relatively easy way to find a lower bound for resistance. Unfortunately, we do not know whether this bound is sharp. We believe that the inequality \eqref{ineqVK} is sharp, while the inequality in \eqref{ineqF} is not.

\subsection{Estimating the value $K_\lam(\vphi, \psi)$}

Now we are going to study the integrand $K_\lam(\vphi, \psi)$. To that end consider the auxiliary problem: given $\varphi,\, \psi \in [0,\, \pi/2]$, find
\beq\label{min fg}
\inf_{\substack{\langle v_1, n_1\rangle = \cos\vphi\\ \langle v_2, n_2\rangle = \cos\psi\\ 
v_1,n_1,v_2,n_2\in S^{d-1}
}}\!\!\!\!\!\!\! \PP(v_1, n_1, v_2, n_2),\ \text{ where} \, \
\PP(v_1, n_1, v_2, n_2) =  -f(|v_1 - v_2|) + g(|n_1 - n_2|)
\eeq
with $f$ and $g$ being strictly monotone increasing functions defined on $\RRR_0^+ := [0,\, +\infty)$. 
The value $K_\lam(\vphi, \psi)$
equals the infimum in \eqref{min fg} when $f(x) = \dfrac{d+1}{4} \Big(\dfrac{x^2}{2}-2\Big)$ and $g(x) = \dfrac{\lam b_{d-1}}{b_d}\, x - \lam$.

\begin{lemma}\label{lfg}
Let the quadruple $\{ v_1,\, n_1,\, v_2,\, n_2 \}$ be a solution of problem \eqref{min fg}.
Then they are coplanar. Moreover, they lie in a closed half-plane\footnote{A set of vectors of the form $\RRR u + \RRR_0^+ u^\bot$, where $u$ and $u^\bot$ are nonzero and orthogonal, and $\RRR_0^+$ is the set of nonnegative real numbers.}, and the order of vectors in a certain direction is $v_1,\, n_1,\, n_2,\, v_2$  (some of the vectors may coincide).
\end{lemma}

\begin{proof}
Consider a minimizer $\{ \bar v_1,\, \bar n_1,\, \bar v_2,\, \bar n_2 \}$. Recall that the angle between $\bar n_1$ and $\bar n_2$ is denoted as $2\eta$, that is, $\eta = \frac 12 \arccos \langle \bar n_1, \bar n_2 \rangle \in [0,\, \pi/2]$. Let us first prove that $\eta \le (\pi - \vphi - \psi)/2$. 

Indeed, assume the opposite and compare $\{ \bar v_1,\, \bar n_1,\, \bar v_2,\, \bar n_2 \}$ with a coplanar admissible quadruple of vectors $\{ v_1,\, n_1,\, v_2,\, n_2 \}$ such that $v_2 = -v_1$, and additionally, the vectors lie in a closed half-plane and follow in the order $v_1,\, n_1,\, n_2,\, v_2$. To find such a quadruple, it suffices to take two mutually orthogonal unit vectors $v_1$ and $v_1^\bot$ and denote $n_1 = \cos\vphi\, v_1 + \sin\vphi\, v_1^\bot$ and $n_2 = -\cos\psi\, v_1 + \sin\psi\, v_1^\bot$. 

One has $|v_1-v_2| = 2 \ge |\bar v_1 - \bar v_2|$ and $0 \le \pi - \vphi - \psi < 2\eta$, hence
$$
|n_1-n_2| = \sqrt{2 - 2\langle n_1,\, n_2 \rangle} =
\sqrt{2 - 2\cos(\pi-\vphi-\psi)} < \sqrt{2 - 2\cos(2\eta)} = |\bar n_1- \bar n_2|.
$$
It follows that $\PP(v_1, n_1, v_2, n_2) < \PP(\bar v_1, \bar n_1, \bar v_2, \bar n_2)$, in contradiction with optimality of $\{ \bar v_1,\, \bar n_1,\, \bar v_2,\, \bar n_2 \}$.

Thus, it is proved that 
$$
2\eta + \vphi + \psi \le \pi.
$$

Select vectors $v_1$ and $v_2$ so as the quadruple $\{ v_1,\, \bar n_1,\, v_2,\, \bar n_2 \}$ is admissible, the four vectors belong to a half-plane and follow in the prescribed order: $v_1,\, \bar n_1,\, \bar n_2,\, v_2$. If $\eta \ne 0$, this pair of vectors $v_1$,\, $v_2$ is unique.

Take the line segment with the endpoints $v_1$ and $v_2$ (that is, the convex hull of $v_1$ and $v_2$). This segment intersects $\mathbb R_0^+ \bar n_1$ and $\mathbb R_0^+ \bar n_2$ at, say, $\bar n_1'$ and $\bar n_2'$, respectively. The length of the segment is equal to $|v_1 - v_2|$ and also to the length of the broken line $\bar v_1 \bar n_1 \bar n_2 \bar v_2$, which in turn is greater than or equal to $|\bar v_1 - \bar v_2|$. Thus, we have $|v_1 - v_2| \ge |\bar v_1 - \bar v_2|$, and therefore, 
\beq\label{PPhi}
\PP(v_1, \bar n_1, v_2, \bar n_2) \le \PP(\bar v_1, \bar n_1, \bar v_2, \bar n_2),
\eeq
and the equality is achieved only when $\bar v_1 = v_1$ and $\bar v_2 = v_2$. Due to optimality of the quadruple $\{ \bar v_1,\, \bar n_1,\, \bar v_2,\, \bar n_2 \}$, the equality in \eqref{PPhi} takes place, and therefore, the vectors of the quadruple belong to a half-plane and follow in the prescribed order.
\end{proof}

Let $\vphi,\, \psi \in [0,\, \pi/2]$. By Lemma \ref{lfg}, if the quadruple $\{ v_1, n_1, v_2, n_2 \}$ minimizes the right hand side of \eqref{minK} and $\eta$ is one half of the angle between $n_1$ and $n_2$ then $\vphi + 2\eta + \psi \le \pi$, and the infimum in \eqref{minK} equals
$$
\frac{d+1}{4}\, \frac{|v_1 + v_2|^2}{2} + \lam \Big( \frac{b_{d-1}}{b_d}\, |n_1 - n_2| - 1 \Big) =
\frac{d+1}{4} \big( 1 + \cos(\vphi + \psi + 2\eta) \big) + \lam\Big( 2\frac{b_{d-1}}{b_d}\, \sin\eta - 1 \Big).
$$
It remains to minimize the latter expression with respect to $\eta$, having fixed $\vphi + \psi = \theta$. Denote
\beq\label{Lam_f}
\Lam = \dfrac{4}{d+1}\, \dfrac{b_{d-1}}{b_d}\, \lam \quad \text{and} \quad f(\eta) = f_{\theta,\Lam}(\eta) = \cos(\theta + 2\eta) + 2\Lam \sin\eta;
\eeq
then problem \eqref{minK} amounts to the following one:
\beq\label{probl2min}
\text{for each} \ \ \theta \in [0,\, \pi] \ \ \text{and} \ \ \Lam > 0 \ \ \text{find} \ \ \varkappa_\Lam(\theta) := \inf_{\eta \in [0,\, (\pi-\theta)/2]} f_{\theta,\Lam}(\eta).
\eeq
We then have
\beq\label{Klamb}
K_\lam(\vphi, \psi) = \dfrac{d+1}{4}\, \Big[ 1 + \varkappa_\Lam(\vphi+\psi) - \dfrac{b_{d}}{b_{d-1}} \Lam \Big].
\eeq 
The function $\varkappa_\Lambda$ will be further evaluated in formulae \eqref{VkLamTheta} and \eqref{etaLam}.

\subsection{Estimating the function $\varkappa_\Lam$ and finishing the proof of Theorem \ref{tt1}}

Now we are going to find the minimizer(s) in \eqref{probl2min} and to describe the function $\varkappa_\Lam$.

If $\theta = 0$ or $\theta = \pi$, the minimizer in \eqref{probl2min} is easy to find. In the former case, the unique minimizer is $\eta = \pi/2$, if $\Lam < 1$, and $\eta = 0$, if $\Lam > 1$, and there are two minimizers $\eta = 0$ and $\pi/2$, if $\Lam = 1$. In the latter case the unique minimizer is $\eta = 0$.

For the intermediate values $0 < \theta < \pi$ the answer is given by the following Lemmas \ref{l m1} and \ref{l MIN}.

Note that if a minimizer $\eta$ lies in the interior of the segment $[0,\, (\pi-\theta)/2]$ then it satisfies the equation $f'(\eta) = 0$, which is equivalent to
\beq\label{eqs}
\Lam = \frac{\sin(\theta + 2\eta)}{\cos\eta}.
\eeq

\begin{lemma}\label{l m1}
If $0 < \Lam \le 1$ and $0 < \theta \le \pi - \arcsin\Lam$, there is a unique solution $\eta$ of equation \eqref{eqs} in the interval $\big[ (\pi/2-\theta)_+,\ (\pi-\theta)/2 \big]$.
\end{lemma}

\begin{proof}
Define the function
$$
\Lam(\eta) = \Lam_\theta(\eta) := \frac{\sin(\theta + 2\eta)}{\cos\eta}
$$
on the interval $\big[0,\, \frac{\pi - \theta}{2} \big]$. Equation \eqref{eqs} is then equivalent to $\Lam(\eta) = \Lam$. The derivative of $\Lam(\eta)$ equals
$$
\Lam'(\eta) = \frac{\gam(\eta)}{\cos^2 \eta},
$$
 where
 $$
 \gam(\eta) = 2\cos(\theta + 2\eta) \cos \eta + \sin(\theta + 2\eta) \sin \eta = \frac{1}{2}\cos(\theta + 3\eta) + \frac{3}{2}\cos(\theta + \eta).
 $$
That is, $\gam(0) = 2\cos\theta$ and $\gam((\pi-\theta)/2) = -2\sin(\theta/2) < 0$, and additionally,
$\gam'(\eta) = -3\sin(\theta + 2\eta) \cos \eta < 0 $ on $\big[0,\, (\pi - \theta)/2\big),$
hence $\gam$ is strictly monotone decreasing.

We have $\Lam(0) = \sin\theta$ and $\Lam\big(\frac{\pi-\theta}{2}\big) = 0$, and additionally,
$$\Lam\Big( \frac{\pi/2-\theta}{3} \Big) = \Lam(\pi/2-\theta) = 1.
$$

Now let $0 < \Lam \le 1$ and consider two cases: $0 < \theta < \pi/2$ and $\pi/2 \le \theta \le \pi - \arcsin\Lam$.
\vspace{2mm}

(i) If $0 < \theta < \pi/2$ then the function $\gam(\eta)$ is first positive and then negative. It follows that the function $\Lam(\eta)$ is first monotone increasing and then monotone decreasing; see Fig.~\ref{figLam}.

Thus, $\Lam(\eta)$ is strictly monotone increasing on $\big[ 0,\, \frac{\pi/2-\theta}{3} \big]$ and strictly monotone decreasing on $\big[ \pi/2-\theta,\, (\pi - \theta)/2 \big]$.
              \begin{figure}[h]
\begin{picture}(0,145)
\scalebox{1}{
\rput(4,0.75){
\psline[linewidth=0.4pt,arrows=->,arrowscale=1.5](0,0)(8,0)
\psline[linewidth=0.4pt,arrows=->,arrowscale=1.5](0,0)(0,4)
\pscurve[linecolor=blue](0,1.8)(1.5,3)(3,3.4)(4.5,3)(7,0)
\psline[linecolor=brown,linewidth=0.4pt,linestyle=dashed](0,3)(4.5,3)(4.5,0)
\psline[linecolor=brown,linewidth=0.4pt,linestyle=dashed](1.5,3)(1.5,0)
\rput(-0.25,3){\scalebox{1}{1}}
\rput(-0.05,-0.25){\scalebox{1}{0}}
\rput(-0.55,1.8){\scalebox{1}{$\sin\theta$}}
\rput(-0.35,3.9){\scalebox{1}{$\Lam$}}
\rput(8.2,0){\scalebox{1}{$\eta$}}
\rput(7,-0.3){\scalebox{1}{$\frac{\pi-\theta}{2}$}}
\rput(4.5,-0.3){\scalebox{1}{$\frac{\pi}{2} - \theta$}}
\rput(1.5,-0.4){\scalebox{1}{$\frac{\frac{\pi}{2} - \theta}{3}$}}
\rput(6.2,2.4){\scalebox{1.2}{$\Lam(\eta)$}}
}}
\end{picture}
\caption{The case $0 < \theta < \pi/2$.}
\label{figLam}
\end{figure}
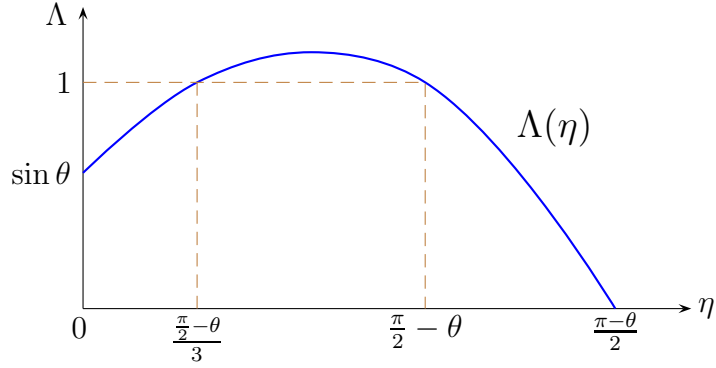 
Hence the equation $\Lam(\eta) = \Lam$ has a unique solution on the interval $\big[ \pi/2 - \theta,\, (\pi - \theta)/2 \big]$; let it be denoted by $\eta = \eta_\Lam(\theta)$. One easily sees that $\eta_1(\theta) = \pi/2 - \theta$ and $\eta_\Lam(\theta) > \pi/2 - \theta$ for $0 < \Lam < 1$.
\vspace{2mm}

(ii) If $\pi/2 \le \theta \le \pi - \arcsin\Lam$ then $\gam(\eta)$ decreases from $\gam(0) = 2\cos\theta < 0$ to $\gam((\pi-\theta)/2) = -2\sin(\theta/2)$, and hence, is negative for $\eta \in [0,\, (\pi-\theta)/2]$.  It follows that the function $\Lam$ is monotone decreasing from $\Lam(0) = \sin\theta$ to $\Lam((\pi-\theta)/2) = 0$.
Since $0 < \Lam \le \sin\theta$, the equation $\Lam(\eta) = \Lam$ has a unique solution on $\big[ 0,\, (\pi - \theta)/2 \big]$, which is also denoted as $\eta_\Lam(\theta)$. In particular, $\eta_\Lam(\pi - \arcsin\Lam) = 0$.
\end{proof}

Now extend the domain of the function $\eta_\Lam$ defined in the proof of this lemma. Namely, let the function $\eta_\Lam : (0,\, \pi) \to \RRR$, $\Lam > 0$ be defined by
\begin{equation}\label{etaLam}
\begin{split}
 \eta_\Lam(\theta) &= \left\{ \begin{array}{ll}  \text{the solution of \eqref{eqs} in $\big[ \big(\frac{\pi}{2}-\theta \big)_+,\ \frac{\pi-\theta}{2} \big]$} ,
 & \text{if}\, \ 0 < \theta < \pi - \arcsin\Lam,\\
  0, & \text{if}\, \ \pi - \arcsin\Lam \le \theta < \pi, \end{array} \right.
\\
& \text{if} \ \ 0 < \Lam \le 1; \\  
\eta_\Lam(\theta) &= 0, \ \ \text{if} \ \ \Lam > 1.
\end{split}
\end{equation}
In particular, $\eta_1(\theta) = (\pi/2 - \theta)_+$.

\begin{lemma}\label{l MIN}
For $\Lam > 0$ and $0 < \theta < \pi$,\, $\eta_\Lam(\theta)$ is a minimizer  in problem \eqref{probl2min}. If $\Lam =1$ and $0 < \theta < \pi/2$, there are two minimizers, $0$ and $\eta_1(\theta) = \pi/2- \theta$. Otherwise, the minimizer $\eta_\Lam(\theta)$ is unique.
\end{lemma}

\begin{proof}
Recall that problem \eqref{probl2min} amounts to minimization of the function $f(\eta) = \cos(\theta + 2\eta) + 2\Lam \sin\eta$ in $[0,\, (\pi-\theta)/2]$. The derivative of $f$ is
$$
f'(\eta) = -2\sin(\theta + 2\eta) + 2\Lam \cos\eta,
$$
and the equation $f'(\eta) = 0$ is equivalent to $\Lam(\eta) = \Lam$. We have $f'((\pi-\theta)/2) = 2\Lam\sin\frac{\theta}{2} > 0$, hence $\eta = (\pi-\theta)/2$ is not a minimizer in Problem \eqref{probl2min}.

Assume that there is a minimizer $\eta_* \ne 0$; then $\eta_*$ lies in the interior of the interval $[0,\, (\pi-\theta)/2]$, and therefore, $f'(\eta_*) = 0$. This implies that $\Lam(\eta_*) = \Lam$, and additionally, $f(\eta_*) \le f(0)$. Thus, $\eta_*$ satisfies the system of equations
 $$
\left\{
\begin{aligned}
\cos(\theta + 2\eta) + 2\Lam \sin\eta \le \cos\theta;\\
\Lam = \frac{\sin(\theta + 2\eta)}{\cos\eta};\\
0 < \eta < \frac{\pi - \theta}{2}.
\end{aligned}
\right.
$$

Let us study this system. After some algebra one derives from the first equation of the system that $\Lam \le \sin(\theta+\eta)$. Using this inequality, one derives from the second equation that $\cos(\theta+\eta) \le 0$, hence
$$
\theta + \eta \ge \pi/2.
$$
Thus, a solution $\eta_*$ of the system satisfies
\beq\label{unique}
\eta_* \in \big[ (\pi/2-\theta)_+,\, (\pi - \theta)/2 \big].
\eeq

The latter inequality also implies that
$$
\Lam = \frac{\sin(\theta + 2\eta_*)}{\cos\eta_*} \le \frac{\sin(\pi/2 + \eta_*)}{\cos\eta_*} = 1;
$$
in other words, if there is a nonzero minimizer of the system then $\Lam \le 1$. On the contrary, if $\Lam > 1$ then the (unique) minimum in \eqref{probl2min} is attained at $\eta = 0$.

Now let $0 < \Lam \le 1$ and consider two cases: $0 < \theta < \pi/2$ and $\pi/2 \le \theta \le \pi$.
\vspace{2mm}

(i) Let $0 < \theta < \pi/2$. The second and third equations of the system are satisfied for $\eta = \eta_\Lam(\theta)$. It remains to check the first equation.

 Indeed, taking $\eta = \eta_\Lam(\theta)$ and using the second equation of the system and the inequality $\theta + \eta \ge \pi/2$ (which becomes equality iff $\Lam = 1$), one obtains
\begin{equation}\label{abc}
\begin{split}
\cos(\theta + 2\eta) + 2\Lam \sin\eta - \cos\theta  = -2\sin(\theta + \eta) \sin\eta + \frac{2\sin(\theta + 2\eta)}{\cos\eta} \sin\eta & \\
 = 2\tan\eta \big[ -\sin(\theta + \eta) \cos\eta + \sin(\theta + 2\eta) \big] = 2\tan\eta\, \sin\eta \cos(\theta + \eta) \le 0, &
\end{split}
\end{equation}
and the equality here is attained iff $\Lam = 1$.

Taking into account inclusion \eqref{unique} and Lemma \ref{l m1}, one concludes that $\eta = \eta_\Lam(\theta)$ is a unique solution of the system, and therefore, is the unique nonzero minimizer of \eqref{probl2min}. Moreover, if $0 < \Lam < 1$, the minimizer is unique, and if $\Lam = 1$, there are two minimizers, $\eta = 0$ and $\eta = \pi/2 - \theta$.
\vspace{2mm}

(ii) If $\pi/2 \le \theta < \pi - \arcsin\Lam$, the function $f$ decreases in $[0,\, \eta_\Lam(\theta)]$ and increases in $[\eta_\Lam(\theta),\, (\pi-\theta)/2]$, hence the unique minimizer in problem \eqref{probl2min} is $\eta_\Lam(\theta)$.
If $\pi - \arcsin\Lam \le \theta < \pi)$, equation \eqref{eqs} does not have nonzero solutions, hence $f$ is monotone increasing, and the unique minimizer in \eqref{probl2min} is $\eta = 0$.
\end{proof}

According to Lemma \ref{l MIN} and formula \eqref{probl2min},
\beq\label{VkLamTheta}
\varkappa_\Lam(\theta) = \cos(\theta + 2\eta_\Lam(\theta)) + 2\Lam \sin\eta_\Lam(\theta).
\eeq

Let now $\Lam \ge1$ \big(correspondingly, $\lam \ge \hat\lam = \hat\lam_d := \dfrac{d+1}{4}\, \dfrac{b_d}{b_{d-1}}$\big). According to Lemma \ref{l MIN}, $\eta = 0$ is a minimizer to Problem \eqref{probl2min}, therefore
$$
\varkappa_\Lam(\theta) = f_{\theta,\Lam}(0) = \cos\theta,
$$
and thus, 
\beq\label{ineqCHI}
K_{\lam}(\vphi, \psi) = \dfrac{d+1}{4} \big[ 1 + \cos(\vphi+\psi) \big] - \lam.  
\eeq
Using equations \eqref{ineqF}, \eqref{ineqVK}, \eqref{ineqCHI}, and \eqref{md}, one obtains
$$
\inf_{\chi\in\Gam_{\varsigma,\varsigma}}  \int_{[0,\, \pi/2]^2} K_{\lam}(\vphi+\psi)\, d\chi(\vphi,\psi)=
\dfrac{d+1}{4} \inf_{\chi\in\Gam_{\varsigma,\varsigma}} \int_{[0,\, \pi/2]^2} (1 + \cos(\vphi+\psi))\, d\chi(\vphi,\psi) - \lam = m_d - \lam,
$$
and
\beq\label{hatlam}
\nvis(D) \ge \lam \DDD +
\inf_{\chi\in\Gam_{\varsigma,\varsigma}}  \int_{[0,\, \pi/2]^2} K_{\lam}(\vphi+\psi)\, d\chi(\vphi,\psi)
= m_d - \lam (1 - \DDD).
\eeq
Taking $\lam = \hat\lam$, one obtains the statement of Theorem \ref{tt1}.

\section{Notes on optimal mass transfer and proof of Theorem \ref{tt2}}\label{sec:proofThm2}

\subsection{Notes on optimal mass transfer}

The following Propositions 1 and 2 are well known in the theory of optimal mass transfer; see, e.\,g., \cite{Ambrosio}.

Let $X \subset \RRR$ be a compact interval. Let $(X, f_0)$ and $(X, f_1)$ be measure spaces with the probability measures $f_0$ and $f_1$. Denote by $\Gam_{f_0,f_1}$ the set of (probability) measures $\gam$ in $X \times X$ whose marginal measures are $f_0$ and $f_1$, that is, satisfying $\gam(A \times X) = f_0(A)$ and $\gam(X \times B) = f_1(B)$ for any Borel subset $A$ of $(X, f_0)$ and any Borel subset $B$ of $(X, f_1)$. Consider a continuous function $c : X \times X \to [0, +\infty)$.

\begin{utv}\label{utv1}
The problem
$$\inf_{\gam\in\Gam_{f_0,f_1}} \int_{X \times X} c(x,y)\, d\gam(x,y)$$
 has a solution.
 \end{utv}

A set $A \subset X \times X$ is said to be $c$-monotone, if for any two points $(x,y)$ and $(x',y')$ in $A$ holds
$$c(x,y) + c(x',y') \le c(x,y') + c(x',y).$$

\begin{utv}\label{utv2}
 If $\gam_*$ minimizes the problem from Proposition 1 then spt$\,\gam_*$ is $c$-monotone.
 \end{utv}

 Let us prove the following simple statement.

\begin{utv}\label{utv3}
Let $c(x,y) = \kap(x+y)$, with $\kap$ being a strictly concave function, and let the measures $f_0$ and $f_1$ coincide and be supported on $X$. Then the support of the minimizer is the diagonal $x=y$, and
\beq\label{mk}
\inf_{\gam\in\Gam_{f_0,f_0}} \int_{X \times X} \kap(x+y)\, d\gam(x,y) = \int_X \kap(2x)\, df_0(x).
\eeq
 \end{utv}

\begin{proof}
Let $\gam_*$ be a minimizer; then for any two points $(x,y)$,\, $(x',y') \in \text{spt}\,\gam_*$ holds $(x'-x)(y'-y) \ge 0$. Indeed, otherwise, denoting $\lam = (y-y')/(x'-x+y-y') \in (0,\, 1)$ and using that $\lam(x+y') + (1-\lam)(x'+y) = x'+y'$ and $(1-\lam)(x+y') + \lam(x'+y) = x+y$, one obtains
$$
\lam g(x+y') + (1-\lam) g(x'+y) < g(x'+y') \quad \text{and} \quad (1-\lam) g(x+y') + \lam g(x'+y) < g(x+y),
$$
hence $$g(x+y') + g(x'+y) < g(x'+y') + g(x+y),$$
in contradiction with $c$-monotonicity of spt$\,\gam_*$.
Thus, for any point $(x_0, y_0) \in \text{spt}\,\gam_*$,
$$\gam_*(\{ x \le x_0,\, y \le y_0 \}) = \gam_*(\{ x \le x_0 \}) = f_0(\{ x \le x_0 \}),
 $$
 and similarly,
 $$\gam_*(\{ x \le x_0,\, y \le y_0 \}) = \gam_*(\{ y \le y_0 \}) = f_0(\{ y \le y_0 \}).
  $$
  It follows that $f_0(\{ x \le x_0 \}) = f_0(\{ y \le y_0 \})$, and since spt$\,f_0 = X$, one obtains $x_0 = y_0$. Thus, spt$\,\gam_*$ is contained in the diagonal $x=y$, and
  $$
  \inf_{\gam\in\Gam_{f_0,f_0}} \int_{X \times X} \kap(x+y)\, d\gam(x,y) =  \int_{X \times X} \kap(x+y)\, d\gam_*(x,y) = \int_X \kap(2x)\, df_0(x).
  $$
  \end{proof}

\subsection{Proof of Theorem \ref{tt2}}

Take $0 < \Lam \le 1$. If $0 < \theta < \pi - \arcsin\Lam$ then, by \eqref{etaLam}, $\eta_\Lam(\theta)$ is the (unique) solution of equation \eqref{eqs} on the interval $\big[ (\pi/2-\theta)_+,\ (\pi-\theta)/2 \big]$. Hence we have $\pi/2 \le \theta + 2\eta_\Lam(\theta) \le \pi$ and $\sin(\theta + 2\eta_\Lam(\theta)) = \Lam \cos\eta_\Lam(\theta) \le \Lam$, and therefore, $\theta + 2\eta_\Lam(\theta) \in [\pi - \arcsin\Lam,\, \pi]$. Thus,
$$
\frac{\pi - \theta - \arcsin\Lam}{2} \le \eta_\Lam(\theta) \le \frac{\pi - \theta}{2}.
$$
Using these inequalities and equation \eqref{VkLamTheta}, one obtains
\beq\begin{split}\label{ineq4}
1 + \varkappa_\Lam(\theta) = 1 + \cos(\theta + 2\eta_\Lam(\theta))  + 2\Lam \sin\eta_\Lam(\theta)\\
\ge 2\Lam \sin\Big( \frac{\pi - \theta - \arcsin\Lam}{2} \Big) = 2\Lam \cos\Big( \frac{\theta + \arcsin\Lam}{2} \Big).
\end{split}\eeq

Define the strictly concave function
$$
\kap(\theta) = \kap_{\Lam,\ve}(\theta) = \left\{ \begin{array}{ll} \cos\Big( \dfrac{\theta + \arcsin\Lam}{2} \Big), & \text{if } \ 0 < \theta < \pi - \arcsin\Lam\\
\dfrac{\pi - \arcsin\Lam - \theta}{2} - \ve(\theta-\pi+\arcsin\Lam)^2, & \text{if } \ \pi - \arcsin\Lam \le \theta <\pi \end{array} \right.;
$$
it is assumed that $\ve > 0$ is small.

Using the relations \eqref{Klamb} and \eqref{ineq4} and the inequality $\kap(\theta) \le \cos\big( \frac{\theta + \arcsin\Lam}{2} \big)$, one obtains
$$
K_\lam(\vphi, \psi) = \dfrac{d+1}{4}\, \big[ 1 + \varkappa_\Lam(\vphi+\psi)  - \dfrac{b_d}{b_{d-1}} \Lam \big]
\ge \dfrac{d+1}{4}\, \Big[ 2\Lam \cos\Big( \frac{\vphi+\psi + \arcsin\Lam}{2} \Big) - \dfrac{b_d}{b_{d-1}} \Lam \Big]
$$ $$
\ge \dfrac{d+1}{4}\, \big[ 2\Lam\, \kap(\vphi+\psi)  - \dfrac{b_d}{b_{d-1}} \Lam \big],
$$
and using this inequality and equations \eqref{ineqF} and \eqref{ineqVK}, one gets
$$
\inf_{D \subset S^{d-1}} \big( \nvis(D) - \lam\DDD \big) \ge 2\Lam\, \dfrac{d+1}{4}
\inf_{\chi\in\Gam_{\varsigma,\varsigma}}  \int_{[0,\, \pi/2]^2} \kap(\vphi+\psi)\, d\chi(\vphi,\psi) - \dfrac{d+1}{4}\, \dfrac{b_d}{b_{d-1}} \Lam.
$$
Denote $a := \arcsin\Lam$. By formula \eqref{mk} in Proposition \ref{utv3},
$$
\inf_{\chi\in\Gam_{\varsigma,\varsigma}}  \int_{[0,\, \pi/2]^2} \kap(\vphi+\psi)\, d\chi(\vphi,\psi) = \int_0^{\pi/2} \kap(2\vphi)\, d\varsigma(\vphi)
$$ $$
=  \int_0^{\frac{\pi}{2}-\frac{{a}}{2}}\! \cos\Big( \vphi + \frac{{a}}{2} \Big) d(\sin^{d-1}\vphi)
 + \int_{\frac{\pi}{2}-\frac{{a}}{2}}^{\pi/2} \! \Big[ \Big( \frac{\pi - {a}}{2} - \vphi \Big) - \ve(2\vphi-\pi+{a})^2 \Big] d(\sin^{d-1}\vphi)
$$
\beq\label{lastint} \begin{split}
&=  \int_0^{\frac{\pi}{2}}\! \cos\Big( \vphi + \frac{{a}}{2} \Big) d(\sin^{d-1}\vphi)\\
 - & \int_{\frac{\pi}{2}-\frac{{a}}{2}}^{\pi/2} \! \Big[ \cos\Big( \vphi + \frac{{a}}{2} \Big) - \Big( \frac{\pi - {a}}{2} - \vphi \Big)  + \ve(2\vphi-\pi+{a})^2 \Big] d(\sin^{d-1}\vphi).
\end{split}\eeq
The former integral in \eqref{lastint} equals
$$\cos({a}/2) \int_0^{\frac{\pi}{2}}\! \cos\vphi\, d(\sin^{d-1}\vphi) - \sin({a}/2) \int_0^{\frac{\pi}{2}}\! \sin\vphi\, d(\sin^{d-1}\vphi)
$$ $$
 = \cos({a}/2)\, \frac{b_d}{2b_{d-1}} - \frac{d-1}{d} \sin({a}/2).
 $$
The integrant in the latter integral in \eqref{lastint} is monotone increasing, and therefore, does not exceed its value at $\pi/2$, which is equal to ${a}/2- \sin({a}/2) + \ve{a}^2$, and therefore, the latter integral does not exceed $({a}/2)({a}/2- \sin({a}/2) + \ve{a}^2)$. Taking into account that $\ve > 0$ is arbitrarily small, one obtains
 $$
 \inf_{\chi\in\Gam_{\varsigma,\varsigma}}  \int_{[0,\, \pi/2]^2} \kap(\vphi+\psi)\, d\chi(\vphi,\psi) \ge
  \cos({a}/2)\, \frac{b_d}{2b_{d-1}} - \frac{d-1}{d} \sin({a}/2) - \frac{{a}}{2} \Big( \frac{{a}}{2}- \sin({a}/2) \Big).
$$
It follows that
\beq\label{in1}
\nvis(D) - \lam\DDD \ge - Q\Big( \dfrac{4}{d+1}\, \dfrac{b_{d-1}}{b_d}\, \lam \Big),
\eeq
where
$$
Q(\Lam) =  \dfrac{d+1}{4} \Big( \frac{b_d}{b_{d-1}} (1 - \cos({a}/2)) + 2\,\frac{d-1}{d} \sin({a}/2) + {a} \big( {a}/2- \sin({a}/2) \big) \Big)  \Lam.
$$
The function $\al \mapsto 1 - \cos(\al/2) - \big( 1 - \frac{1}{\sqrt{2}} \big) \sin\al$,\, $0 \le \al \le \pi/2$ is strictly convex and is equal to zero at $\al=0$ and $\al=\pi/2$, hence it is negative for $0 < \al </2$. Substituting $\al = a$ and using that $\sin{a} = \Lam$, one obtains
\beq\label{ineq1}
 1 - \cos({a}/2) \le \Big( 1 - \frac{1}{\sqrt{2}} \Big) \Lam.
 \eeq
Since $\cos({a}/2) \ge 1/\sqrt 2$, one has $\sin({a}/2) \le \sqrt 2 \sin({a}/2)\cos({a}/2) = (1/\sqrt{2}) \sin{a}$, that is,
\beq\label{ineq2}
\sin({a}/2) \le \frac{1}{\sqrt{2}}\, \Lam.
\eeq
Further, the function $\al \mapsto \al \big( \al/2- \sin(\al/2) \big) - \frac{\pi}{2} \Big( \frac{\pi}{4} - \frac{1}{\sqrt{2}} \Big) \sin\al$,\, $0 \le \al \le \pi/2$ is strictly convex and is equal to zero at $\al=0$ and $\al=\pi/2$, hence it is negative for $0 < \al </2$. Substituting $\al = a$ and using that $\sin{a} = \Lam$, one gets
\beq\label{ineq3}
{a} \big( {a}/2- \sin({a}/2) \big) \le \frac{\pi}{2} \Big( \frac{\pi}{4} - \frac{1}{\sqrt{2}} \Big) \Lam.
\eeq
Using inequalities \eqref{ineq1}, \eqref{ineq2}, and \eqref{ineq3}, one obtains that $Q(\Lam) \le \dfrac{c}{2} \lam^2,$ where $\Lam = \dfrac{4}{d+1}\, \dfrac{b_{d-1}}{b_d}\, \lam$ and
$$
c = \dfrac{8}{d+1}  \Big( \frac{b_{d-1}}{b_{d}} \Big)^2\Big[ \frac{b_d}{b_{d-1}} \Big( 1 - \frac{1}{\sqrt{2}} \Big) + \sqrt{2}\,\frac{d-1}{d} + \frac{\pi}{2} \Big( \frac{\pi}{4} - \frac{1}{\sqrt{2}} \Big) \Big],$$
hence
$$
\nvis(D) - \lam\DDD \ge - \dfrac{c}{2} \lam^2.
$$

Taking $\lam = \DDD/c$, one comes to the inequality
$$
\nvis(D) \ge \frac{1}{2c}\, \DDD^2.
$$
Claim (a) of Theorem \ref{tt2} is proved.

Further, we have
$$
Q\Big( \dfrac{4}{d+1}\, \dfrac{b_{d-1}}{b_d}\, \lam \Big) = \dfrac{4(d - 1)}{d(d+1)}\, \Big( \dfrac{b_{d-1}}{b_d} \Big)^2\, \lam^2\, (1 + o(1))
\quad \text{as} \ \ \lam \to 0.
$$
Substituting $\lam = \dfrac{d(d+1)}{8(d - 1)}\, \Big( \dfrac{b_{d}}{b_{d-1}} \Big)^2 \DDD$ in \eqref{in1} one obtains
$$
\nvis(D) \ge \dfrac{d(d+1)}{16(d - 1)}\, \Big( \dfrac{b_{d}}{b_{d-1}} \Big)^2 \DDD^2 (1 + o(1)) \quad \text{as} \ \ \DDD \to 0.
$$
Claim (b) of Theorem \ref{tt2} is also proved.

\section{Numerical simulation of the problem}\label{sec:numeric}

By \eqref{ineqF} and \eqref{ineqVK} we have
$$
\nvis(D) \ge \lam\DDD + \III(\lam),
$$
where
$$
\III(\lam) = \III_d(\lam) :=\inf_{\chi\in\Gam_{\varsigma,\varsigma}}  \int_{[0,\, \pi/2]^2} K_{\lam}(\vphi,\psi)\, d\chi(\vphi,\psi)
$$
(recall that $K_\lam = K_\lam^d$ and $\varsigma = \varsigma_d$ depend on $d$).
According to \eqref{Klamb},
\begin{equation}\label{infChi}
\III(\lam) = \frac{d+1}{4} \inf_{\chi\in\Gam_{\varsigma,\varsigma}}  \int_{[0,\, \pi/2]^2} [1 + \varkappa_\Lam(\vphi+\psi)]\, d\chi(\vphi,\psi) - \lam;
\end{equation}
recall that $\Lam = \dfrac{4}{d+1}\, \dfrac{b_{d-1}}{b_d}\, \lam$ and the function $\varkappa_\Lam$ is given by \eqref{VkLamTheta}.

Denote by $\III^*$ the Legendre transform of $\III$,\, $\III^*(x) = \sup_{\lam>0} \big( \lam x + \III(\lam) \big)$. For $\lam \ge \hat\lam = \dfrac{d+1}{4}\, \dfrac{b_d}{b_{d-1}}$ and $0 \le x \le 1$ by \eqref{hatlam} one obtains
$$
\lam x + \III(\lam) = m_d - \lam(1 - x) \le  m_d - \hat\lam(1 - x) = \hat\lam x + \III(\hat\lam),
$$
hence
\beq\label{istar}
\III^*\big(x\big) = \sup_{0<\lam\le\hat\lam} \big( \lam x + \III(\lam) \big).
\eeq

So, one comes to the following theorem.

\begin{theorem}\label{t5}
We have
\begin{equation}\label{ineq5}
\nvis(D) \ge \III^*(\DDD).
\end{equation}
\end{theorem}

Note in passing that for $0 < \Lam < 1$ the function $\varkappa_\Lam(\theta)$ is strictly concave for $\theta \in [0,\, \pi - \arcsin\Lam]$ and strictly convex for $\theta \in [\pi - \arcsin\Lam,\, \pi]$; for the proof see the Appendix. Numerical plots of $1+\kappa_{\Lambda}$ are shown in Fig.~\ref{fig:kappamanylambda}.  It follows that the measure $\chi_*$ minimizing the integral in \eqref{infChi} satisfies the following conditions: the sets spt$\,\chi_* \cap \{ \vphi+\psi \le \pi - \arcsin\Lam \}$ and spt$\,\chi_* \cap \{ \vphi+\psi \ge \pi - \arcsin\Lam \}$ are the graphs of (generally multivalued) monotone increasing and monotone decreasing mappings, respectively, that is, any two points $(\vphi_1,\psi_1)$ and $(\vphi_2,\psi_2)$ from the former set satisfy the inequality $(\vphi_1 - \vphi_2)(\psi_1 - \psi_2) \ge 0$, and any two points from the latter one satisfy the inverse inequality.

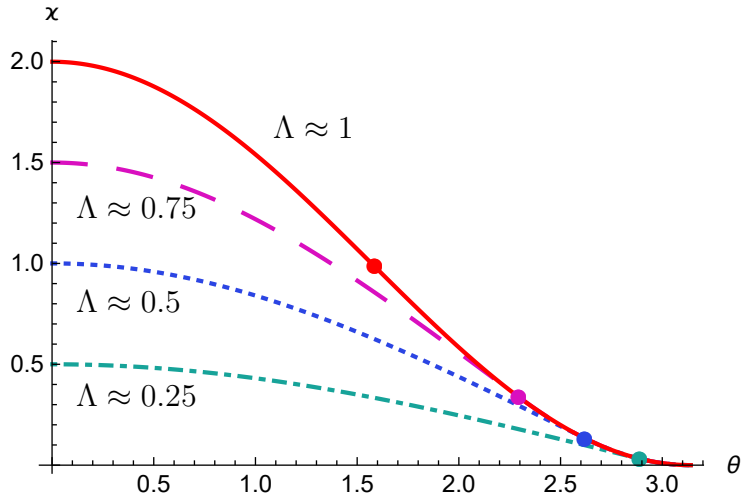
\begin{figure}[ht]
\centering
\begin{overpic}[
width=0.6\textwidth]{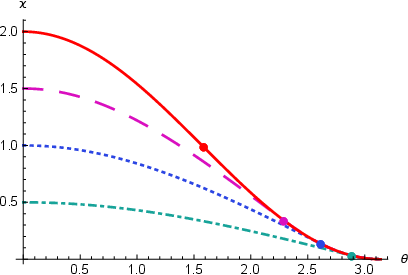}
\put(9,12){$\Lambda \approx 0.25 $}
\put(9,25){$\Lambda \approx 0.5 $}
\put(9,38){$\Lambda \approx 0.75 $}
\put(36,49){$\Lambda \approx 1 $}
\end{overpic}
\caption{Numerical plots of the function $1+\varkappa_\Lam(\theta)$ for fixed values of $\Lambda$. The dots show the inflection points on the plot for $\theta = \pi - \arcsin \Lambda$: lower values of $\Lambda$ correspond to the inflection points with higher $\theta$-coordinates.}
\label{fig:kappamanylambda}
\end{figure}

To evaluate the value
\begin{equation}\label{eq:valest}
\inf_{\chi\in\Gam_{\varsigma,\varsigma}}  \int_{[0,\, \pi/2]^2} [1 + \varkappa_\Lam(\vphi+\psi)]\, d\chi(\vphi,\psi)    
\end{equation}
approximately, we discretize the problem using linear programming. We have
\begin{align}\label{eq:LP}
    \min  & \sum_{i=1}^n \sum_{j=1}^n c_{ij} x_{ij}\notag\\
    \text{s.t.} & \sum_{i=1}^n x_{ik} =\sum_{j=1}^n x_{kj} = b_k \quad \forall k \in \{1,\dots, n\},\\
    & x_{ij}\geq 0\; \, \forall\, i, j\in \{1,\dots, n\}.\notag
\end{align}

Here
\[
b_k = \sin^{d-1} \vphi_k - \sin^{d-1}\vphi_{k-1} \quad \forall k \in \{1,\dots, n\}, \quad \text{with} \ \ \vphi_k=\frac{k}{n}\cdot \frac{\pi}{2},
\] and for $c_{ij}$'s we consider two choices. In the first instance we consider 
\[
c_{ij} = \inf \{ 1 + \varkappa_\Lam\left(\vphi+\psi \right) : (\vphi, \psi)  \in [\vphi_{i-1},\, \vphi_i] \times [\vphi_{j-1},\, \vphi_j] \}\quad \forall i, j\in \{1,\dots, n\}.
\]
Since by \eqref{derivative}, the function $\varkappa_\Lam$ is monotone decreasing, in this case one has
$$
c_{ij} = 1 + \varkappa_\Lam\left(\frac{i+j}{n}\cdot \frac{\pi}{2} \right).
$$
An exact solution to the linear program \eqref{eq:LP} with this choice of objective function coefficients gives a lower approximation to the value of the integral \eqref{eq:valest}. 

Another choice is to consider the values of $c_{ij}$ that correspond to the values of $1+ \varkappa_\Lam $ in the middle of each cell, that is,
$$
c_{ij} = 1 + \varkappa_\Lam\left(\frac{i+j-1}{n}\cdot \frac{\pi}{2} \right).
$$

We ran our numerical experiments in python and used Gurobi package to solve the Linear Programming problem. The numerical plots shown in Fig.~\ref{figGraphs} are obtained from solving the linear programs with $1000\times 1000$ variables for 1000 values of $\lambda$, for $d=2$ and $d=3$ (in this case we used the middle-of-cell version). 

The solutions of our linear programs (that is, the nonzero values of $x_{ij}$ that correspond to the optimal solutions) are plotted in Figures~\ref{fig:measures2d}--\ref{fig:measures3dproj} for selected values of $\lambda$ (that are close to the values $1/3$, $2/3$ and $1$ for $\Lambda$). These plots were obtained by solving the linear programming problems on a $4000\times 4000$ grid (again we used the midpoint values). It took around 16 minutes on average to solve each one of the $4000\times 4000$ linear programming problems on a desktop computer (Intel Core i9-10900 CPU  2.80GHz 64 bit Windows PC, with Gurobi optimizer version 10.0.0), in case some readers find this information relevant.
\begin{figure}[ht]
\centering
\begin{overpic}[
width=0.3\textwidth]{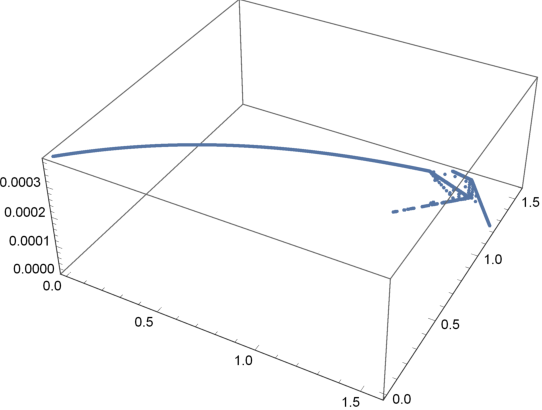}
\put(0,65){$d=2$, $\lambda\approx 0.3888 $}
\end{overpic}
\;
\begin{overpic}[
width=0.3\textwidth]{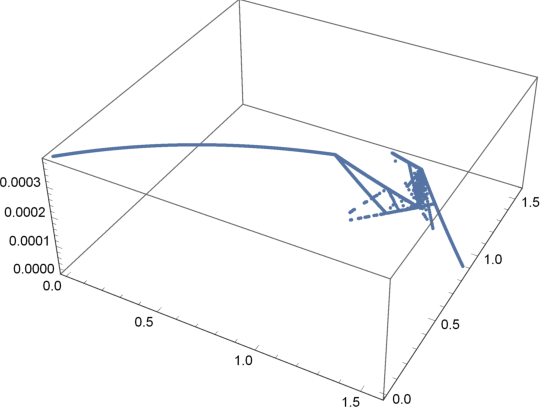}
\put(0,65){$d=2$, $\lambda\approx 0.7775 $}
\end{overpic}
\;
\begin{overpic}[
width=0.3\textwidth]{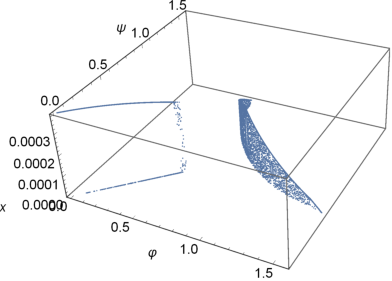}
\put(0,65){$d=2$, $\lambda\approx 1.1781 $}
\end{overpic}
\caption{Numerical solutions showing the approximation of optimal measures $\chi (\varphi,\psi)$ in the two-dimensional case: the horizontal plane corresponds to the (symmetric) parameters $\varphi$ and $\psi$, and the vertical axis shows nonzero values of approximating variables $x_{ij}$.}
\label{fig:measures2d}
\end{figure}

\begin{figure}[ht]
\centering
\begin{overpic}[
width=0.3\textwidth]{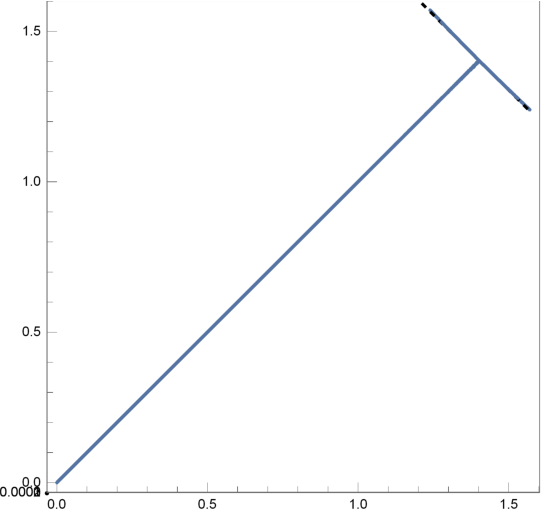}
\put(30,10){$d=2$, $\lambda\approx 0.3888 $}
\end{overpic}
\;
\begin{overpic}[
width=0.3\textwidth]{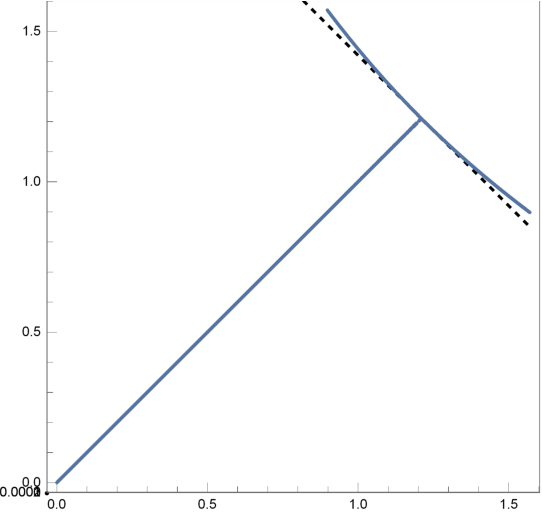}
\put(30,10){$d=2$, $\lambda\approx 0.7775 $}
\end{overpic}
\;
\begin{overpic}[
width=0.3\textwidth]{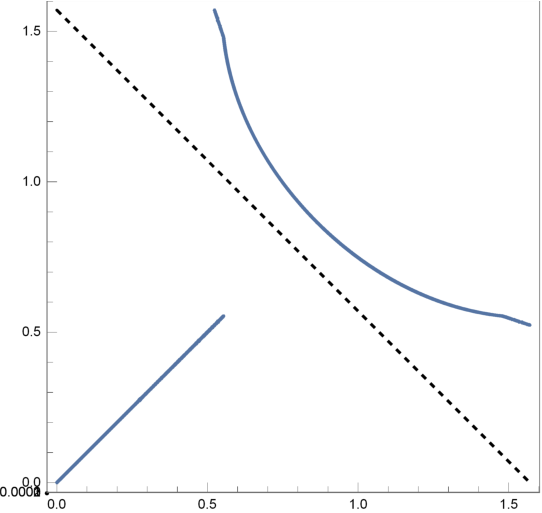}
\put(30,10){$d=2$, $\lambda\approx 1.1781 $}
\end{overpic}
\caption{These images are projections of the graphs shown in Fig.~\ref{fig:measures2d} onto the $(\varphi,\psi)$ plane. The dashed lines are the lines $\varphi + \psi = \pi - \arcsin\Lambda$.
}
\label{fig:measures2dproj}
\end{figure}

\begin{figure}[ht]
\centering
\begin{overpic}[
width=0.3\textwidth]{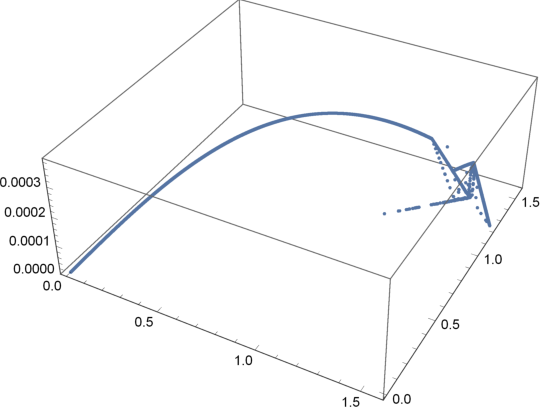}
\put(0,65){$d=3$, $\lambda\approx 0.4400 $}
\end{overpic}
\;
\begin{overpic}[
width=0.3\textwidth]{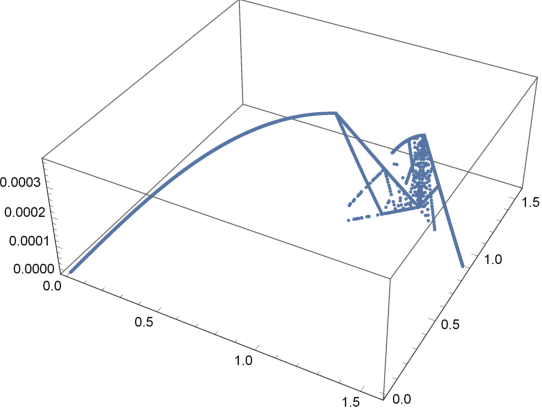}
\put(0,65){$d=3$, $\lambda\approx 0.8800 $}
\end{overpic}
\;
\begin{overpic}[
width=0.3\textwidth]{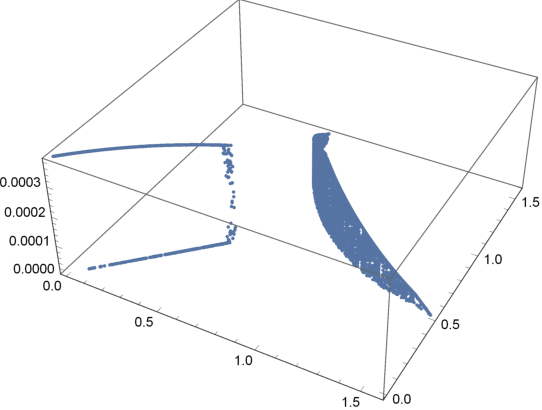}
\put(0,65){$d=3$, $\lambda\approx 1.1333 $}
\end{overpic}
\caption{Numerical solutions showing the approximation of optimal measures $\chi (\varphi,\psi)$ in the three-dimensional case, similarly to the two-dimensional case: the horizontal plane corresponds to the (symmetric) parameters $\varphi$ and $\psi$, and the vertical axis shows nonzero values of approximating variables $x_{ij}$.
}
\label{fig:measures3d}
\end{figure}

\begin{figure}[ht]
\centering
\begin{overpic}[
width=0.3\textwidth]{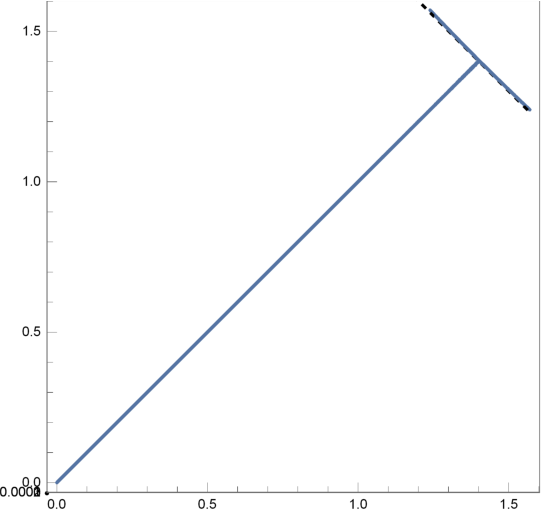}
\put(30,10){$d=3$, $\lambda\approx 0.4440 $}
\end{overpic}
\;
\begin{overpic}[
width=0.3\textwidth]{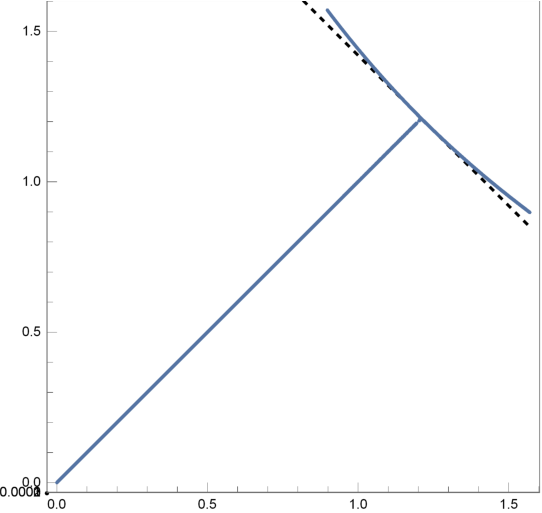}
\put(30,10){$d=3$, $\lambda\approx 0.8880 $}
\end{overpic}
\;
\begin{overpic}[
width=0.3\textwidth]{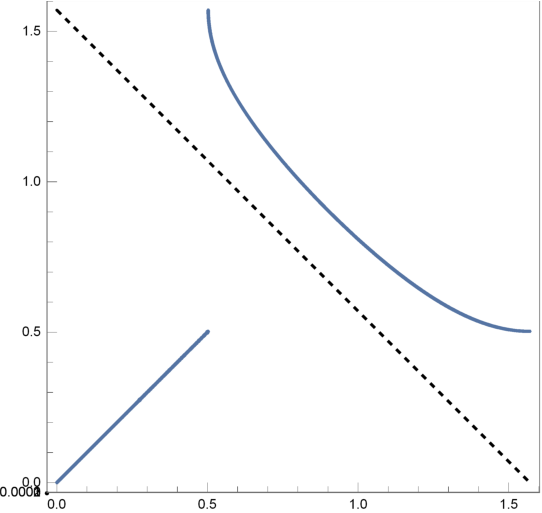}
\put(30,10){$d=3$, $\lambda\approx 1.1333 $}
\end{overpic}
\caption{These images are projections of the graphs shown in Fig.~\ref{fig:measures3d} onto the $(\varphi,\psi)$ plane. The dashed lines are the lines $\varphi + \psi = \pi - \arcsin\Lambda$
}
\label{fig:measures3dproj}
\end{figure}
These values may provide useful intuition about the optimal measure. It appears that the measure concentrates on a low-dimensional subset, and the peculiar patterns seen on the plots are likely artifacts of the numerical model that we use. It is a tempting direction for future research and numerical experiments to obtain a more consistent and accurate approximation to the measure itself and generally sharpen our computational results. Our code is posted on GitHub, \cite{code}.

The numerically found value $\III^*(1)$ equals $0.9878238$ for $d=2$ and $0.9694462$ for $d=3$ (calculated using the LP approximation on $4000\times 4000$ grid and with midcell value choice for $c_{ij}$), which is a good approximation of $m_2$ and $m_3$, respectively. The difference between the numerically obtained approximation and earlier theoretical estimate shows in the seventh decimal place, and is therefore smaller than $10^{-6}$. The lower approximation (that is, the same calculation on $4000\times 4000$ grid, but with the minimal values assigned to $c_{ij}$) resulted in $0.9876127$ and $0.9691021$ for $m_2$ and $m_3$ respectively. 


\section*{Appendix}

Here we prove that for $0 < \Lam < 1$ the function $\varkappa_\Lam(\theta)$ is (a) strictly concave for $\theta \in [0,\, \pi - \arcsin\Lam]$ and (b) strictly convex for $\theta \in [\pi - \arcsin\Lam,\, \pi]$.

(a) It suffices to show that $\frac{d^2}{d^2 \theta}\varkappa_\Lam(\theta) < 0$ for $0 < \theta < \pi - \arcsin\Lam$. One easily derives from \eqref{eqs} the following relations for the function $\eta = \eta_\Lam(\theta)$:
$$
\eta' = -\frac{\cos(\theta + 2\eta)}{\Lam\sin\eta + 2\cos(\theta + 2\eta)}
\quad \text{and} \quad 1 + 2\eta' = \frac{\Lam\sin\eta}{\Lam\sin\eta + 2\cos(\theta + 2\eta)},
$$
whence
\begin{equation}\label{derivative}\begin{split}
\frac{d}{d\theta}\varkappa_\Lam(\theta) = -\sin(\theta + 2\eta) (1+2\eta') +2\Lam \eta' \cos\eta
\\
= - \frac{\Lam\sin\eta \sin(\theta + 2\eta) + 2\Lam\cos\eta \cos(\theta + 2\eta)}{\Lam\sin\eta + 2\cos(\theta + 2\eta)}
= -\Lam\cos\eta,
\end{split}\end{equation}
and
$$
\frac{d^2}{d^2 \theta}\varkappa_\Lam(\theta) = \Lam \eta' \sin\eta = -\Lam \frac{\sin\eta\cos(\theta + 2\eta)}{\Lam\sin\eta + 2\cos(\theta + 2\eta)}.
$$
Let us show that both numerator and denominator in the right hand side of this equation are negative; it will follow that $\frac{d^2}{d^2 \theta}\varkappa_\Lam(\theta) < 0$.

One has $0 \le \eta = \eta_\Lam(\theta) < \frac{\pi-\theta}{2}$. Taking account of \eqref{eqs}, the equality $\eta = 0$ implies that $\Lam = \sin\theta$ and $\theta \ge \pi/2$, hence $\theta = \pi - \arcsin\Lam$ which is impossible. It follows that $\sin\eta > 0$.

Further, one has $\theta + 2\eta \le \pi$ and $\theta + 2\eta \ge \theta + 2\big(\frac{\pi}{2}-\theta \big)_+$. In particular, if $\theta < \pi/2$ then $\theta + 2\eta \ge \pi - \theta > \pi/2$, and if $\theta \ge \pi/2$ then $\theta + 2\eta \ge \theta$. Thus, $\pi/2 \le \theta + 2\eta \le \pi$, and the equality $\theta + 2\eta = \pi/2$ implies that $\theta = \pi/2$ and $\eta = 0$, hence by \eqref{eqs}, $\Lam = 1$, which is impossible. It follows that $\cos(\theta + 2\eta) < 0$.

Finally, since $\pi/2 \le \theta + \big(\frac{\pi}{2}-\theta \big)_+ \le \theta + \eta \le \pi$, by \eqref{eqs} one has
$$
\Lam\sin\eta + 2\cos(\theta + 2\eta) = \frac{\sin\eta \sin(\theta + 2\eta) + 2\cos\eta \cos(\theta + 2\eta)}{\cos\eta} < \frac{\cos(\theta + \eta)}{\cos\eta} \le 0.
$$

(b) For $\theta \in [\pi - \arcsin\Lam,\, \pi]$, by \eqref{etaLam} we have $\eta_\Lam(\theta) = 0$, and so, by \eqref{VkLamTheta}, $\varkappa_\Lam(\theta) = \cos\theta$; therefore the function $\varkappa_\Lam$ is strictly convex on this interval.

Let us note in passing that according to \eqref{etaLam}, for $\theta \in (0,\, \pi)$ we have $\eta = \eta_\Lam(\theta) \in [0,\, \pi/2)$, and therefore by formula \eqref{derivative}, 
\begin{equation}\label{neg}
    \frac{d}{d\theta}\varkappa_\Lam(\theta) < 0.
\end{equation}

\section*{Acknowledgements}

The work of AP was supported by the projects UIDB/04106/2020 and UIDP/04106/2020, https://doi.org/10.54499/UIDB/04106/2020 and https://doi.org/10.54499/UIDP/04106/2020, and by the project CoSysM3, ref. 2022.03091.PTDC through FCT.

\end{document}